\documentclass[a4paper,12pt]{article}

\usepackage{amsmath,amssymb,amsthm, mathrsfs}
\usepackage{latexsym}
\usepackage{graphics}
\usepackage{hyperref}
\usepackage{graphicx,psfrag}
\usepackage[bf, small]{titlesec}
\usepackage{authblk} 
\usepackage{soul}

\usepackage{lineno}

\theoremstyle{plain}
\newtheorem{Thm}{Theorem}[section]
\newtheorem{Lem}[Thm]{Lemma}
\newtheorem{Prop}[Thm]{Proposition}
\newtheorem{Cor}[Thm]{Corollary}

\theoremstyle{definition}

\usepackage{standalone}
\usepackage{pgfpages,tikz,pgfkeys,pgfplots}
\usetikzlibrary{arrows,positioning,matrix,fit,backgrounds,shapes,shapes.geometric, intersections}

\usetikzlibrary{shapes.callouts,decorations.pathmorphing} 
\usetikzlibrary{shadows} 
\usepackage{calc} 
\usetikzlibrary{decorations.markings} 
\tikzstyle{vertex}=[circle, draw, inner sep=0pt, minimum size=6pt] 
\newcommand{\vertex}{\node[vertex]}
\usetikzlibrary{arrows,matrix} 

\newcommand{\NNN}{\mathcal{N}}

\title{The niche graphs of multipartite tournaments}

\author[1]{Soogang Eoh
}
\author[1]{Myungho Choi
}
\author[1]{Suh-Ryung Kim
}

\affil[1]{Department of Mathematics Education,
Seoul National University, Seoul 08826, Republic of Korea}

\begin{document}
\maketitle
\begin{abstract}
The  niche graph of a digraph $D$, denoted by $\NNN(D)$, has  $V(D)$ as the vertex set and  an edge $uv$ if and only if $(u,w) \in A(D)$ and $(v,w) \in A(D)$, or $(w,u) \in A(D)$ and $(w,v) \in A(D)$ for some $w \in V(D)$. The notion of niche graph was introduced by Cable~{\it et al.}~\cite{Cable} as a variant of competition graph. If a graph is the niche graph of a digraph $D$, it is said to be niche-realizable through $D$. If a graph $G$ is niche-realizable through a $k$-partite tournament for an integer $k \ge 2$, then we say that the pair $(G, k)$ is niche-realizable. Bowser~{\it et al.}~\cite{bowser1999niche} studied the graphs that are niche-realizable through a tournament and Eoh~{\it et al.}~\cite{eoh2018niche} studied niche-realizable pairs $(G, k)$ for $k=2$.
In this paper, we study niche-realizable pairs $(G, k)$ when $G$ is a graph and $k$ is an integer at least $3$ to extend their work. We show that the niche graph of a $k$-partite tournament has at most three components if $k \ge 3$ and is connected if $k \ge 4$. Then we find all the  niche-realizable pairs $(G, k)$ when $G$ is a disconnected graph, when $G$ is a complete graph, and when $G$ is a connected triangle-free graph.
\end{abstract}


\section{Introduction}
In this paper, a graph means a simple graph. For all undefined graph theory terminology, see~\cite{bondy}.

Cohen~\cite{cohen} introduced the notion of competition graph while studying predator-prey concepts in ecological food webs.
The \emph{competition graph} of a digraph $D$
is the graph having the vertex set $V(D)$ and an edge $uv$ if and only if $(u,w) \in A(D)$ and $(v,w) \in A(D)$ for some $w \in V(D)$.
Cohen's empirical observation that real-world competition graphs are usually interval graphs had led to a great deal of research on the structure of competition graphs and on the relation between the structure of digraphs and their corresponding competition graphs. In the same vein, various variants of competition graph have been introduced and studied, one of which is the notion of niche graph  introduced by Cable~{\it et al.}~\cite{Cable} (see \cite{ bowser1991some, bowser1999niche, m-step, tournament, fishburn1992niche, p-competition, RobertsSheng, Scott, seager1998cyclic} for various variants of competition graph).

The  {\em niche graph} of a digraph $D$, denoted by $\NNN(D)$, has $V(D)$ as the vertex set and an edge $uv$ if and only if $(u,w) \in A(D)$ and $(v,w) \in A(D)$, or $(w,u) \in A(D)$ and $(w,v) \in A(D)$ for some $w \in V(D)$.
If a graph is the niche graph of a digraph $D$, then it is said to be \emph{niche-realizable through} $D$.
If a graph $G$ is niche-realizable through a $k$-partite tournament for an integer $k \ge 2$, then we say that the pair $(G, k)$ is \emph{niche-realizable} for notational convenience.

Bowser~{\it et al.}~\cite{bowser1999niche} studied the graphs that are niche-realizable through a tournament and Eoh~{\it et al.}~\cite{eoh2018niche} studied the graphs that are niche-realizable through a bipartite tournament.
We extend their work by studying niche-realizable pairs $(G, k)$ for a graph $G$ and an integer $k \ge 3$.

We first show that the niche graph of a $k$-partite tournament is connected if $k \ge 4$ and has at most three components if $k \ge 3$   (Theorem~\ref{thm:K4 connected} and Corollary~\ref{cor:the number of components}). Then we find all the niche-realizable pairs $(G,k)$  when $G$ is disconnected (Theorems~\ref{thm:3components} and \ref{thm:2components}). We show that the niche graph of a $k$-partite tournament contains no induced path of length $5$ (Theorem~\ref{thm:diameter}). Finally, we find all  the niche-realizable pairs $(G,k)$ when $G$ is a complete graph and when $G$ is a connected triangle-free graph (Theorems~\ref{thm:complete} and \ref{thm:triangle-free}).

\section{Preliminaries}
For a digraph $D$, a digraph is said to be the \emph{converse} of $D$ and denoted by $D^{\leftarrow}$ if its vertex set is $V(D)$ and its arc set is $\{(u, v) \mid (v, u) \in A(D)\}$.

By the definition of niche graph, the following lemmas are immediately true.
\begin{Lem}\label{lem:inverse}
For a digraph $D$, the niche graph of $D$ and the niche graph of $D^{\leftarrow}$ are the same.
\end{Lem}

\begin{Lem}\label{lem:subgraph}
Let $D$ be a digraph and $D'$ be a subdigraph of $D$.
Then the niche graph of $D'$ is a subgraph of the niche graph of $D$.
\end{Lem}

\begin{Lem} \label{lem:in-out-degree-condition}
For a digraph $D$, if the niche graph of $D$ is $K_m$-free, then $  d^+_D(u) \leq m-1$ and $ d^-_D(u)\leq m-1$ for each vertex $u$ in $D$.
\end{Lem}

It is easy to check that the following lemma is true.
\begin{Lem}\label{lem:K3}
Let $D$ be an orientation of $K_3$.
Then the niche graph of $D$ is isomorphic to
\[
\begin{cases}
  I_3 & \mbox{if $D$ is a directed cycle};  \\
  P_3 & \mbox{otherwise}.
\end{cases}
\]
\end{Lem}

Bowser~{\it et al.}~\cite{bowser1999niche} have  shown that the complement of the niche graph of a tournament is one of the following: a cycle of odd order, a path
of even order, a forest of odd order consisting of two paths, a forest of even order
consisting of three paths, or a forest of four or more paths.
By this result, we have the following lemma.

\begin{Lem}\label{lem:K4}
The niche graph of an orientation of $K_4$ is connected.
\end{Lem}

\begin{Thm}\label{thm:K4 connected}
For $k \ge 4$, the niche graph of a $k$-partite tournament is connected.
\end{Thm}
\begin{proof}
Let $G$ be the niche graph of the $k$-partite tournament $D$.
We denote the partite sets of $D$ by $(X_1, X_2, \ldots, X_k)$.
Take two vertices $x$ and $y$ in $G$.
It suffices to show that $x$ and $y$ are connected in $G$.

Suppose that $x$ and $y$ belong to different partite sets in $D$.
Without loss of generality, we may assume that $x \in X_1$ and $y \in X_2$.
Since $k \ge 4$, we may take $z \in X_3$ and $w \in X_4$.
Let $D_1$ be the subdigraph of $D$ induced by $\{x, y, z, w\}$.
Then $D_1$ is an orientation of $K_4$.
Thus, by Lemma~\ref{lem:K4}, the niche graph of $D_1$ is connected.
By Lemma~\ref{lem:subgraph}, the niche graph of $D_1$  is a subgraph of $G$ and so $x$ and $y$ are connected in $G$.

Now suppose that $x$ and $y$ belong to the same partite set in $D$.
Then, without loss of generality, we may assume that $\{x, y\} \subset X_4$.
Take a vertex $z$ in $X_i$ for some $i \in \{1,2,3 \}$.
Since $x$ (resp.\ $y$) and $z$ belong to different partite set in $D$, $x$ (resp.\ $y$) and $z$ are connected in $G$ by the previous argument.
Therefore $x$ and $y$ are connected in $G$.
\end{proof}

A \emph{stable set} of a graph is a set of vertices no two of which are adjacent. A stable set in a graph is \emph{maximum} if the graph contains no larger stable set. The cardinality of a maximum stable set in a graph $G$ is called the \emph{stability number} of $G$, denoted by $\alpha(G)$.

\begin{Thm}\label{thm:stability number}
For $k \ge 3$, the niche graph of a $k$-partite tournament has stability number at most $3$.
\end{Thm}
\begin{proof}
Let $G$ be the niche graph of a $k$-partite tournament $D$.
Suppose, to the contrary, $\alpha(G) \ge 4$.
Then we may take a stable set of size $4$ in $G$.
We denote it by $\{x_1, x_2, x_3, x_4\}$.

Suppose that there exist partite sets $X_1$ and $X_2$ of $D$ such that $\{x_1, x_2, x_3, x_4\} \subset X_1 \cup X_2$.
Since $k \ge 3$, we may take a vertex $x_5$ in a partite set $X_3$ of $D$ distinct from $X_1$ and $X_2$.
Since $D$ is a $k$-partite tournament, $\{x_1, x_2, x_3, x_4\} \subset N_{D}^+(x_5) \cup N_{D}^-(x_5)$.
Therefore $|N_{D}^+(x_5) \cap \{x_1, x_2, x_3, x_4\}| \ge 2$ or $|N_{D}^-(x_5) \cap \{x_1, x_2, x_3, x_4\}| \ge 2$.
Yet, each of $N_{D}^+(x_5) \cap \{x_1, x_2, x_3, x_4\}$ and $N_{D}^-(x_5) \cap \{x_1, x_2, x_3, x_4\}$ forms a clique in $G$, which is a contradiction to the assumption that  $\{x_1, x_2, x_3, x_4\}$ is a stable set of $G$.
Hence there are three elements in $\{x_1, x_2, x_3, x_4\}$ belonging to distinct partite sets.
Then there is a partite set $X$ satisfying $|X \cap \{x_1, x_2, x_3, x_4\}|=1$.
Without loss of generality, we may assume that $X \cap \{x_1, x_2, x_3, x_4\} = \{x_4\}$.
Then $\{x_1, x_2, x_3\} \subset N_{D}^+(x_4) \cup N_{D}^-(x_4)$ and so $|N_{D}^+(x_4) \cap \{x_1, x_2, x_3\}| \ge 2$ or $|N_{D}^-(x_4) \cap \{x_1, x_2, x_3\}| \ge 2$.
Since each of $N_{D}^+(x_4) \cap \{x_1, x_2, x_3\}$ and $N_{D}^-(x_4) \cap \{x_1, x_2, x_3\}$ forms a clique in $G$, $\{x_1, x_2, x_3\}$ cannot be a stable set of $G$, which is a contradiction.
This completes the proof.
\end{proof}

\noindent
From the above theorem, the following corollary immediately follows.

\begin{Cor}\label{cor:the number of components}
For $k \ge 3$, the niche graph of a $k$-partite tournament has at most three components.
\end{Cor}

\section{Niche-realizable pairs $(G,k)$ when $G$ is disconnected}
In this section, we completely characterize the niche graphs of $k$-partite tournaments for $k \ge 3$ which are disconnected.

Theorem~\ref{thm:K4 connected} tells us that, for a disconnected graph $G$ and $k \ge 3$, if $(G, k)$ is niche-realizable, then $k=3$.
In addition, the niche graph of a $k$-partite tournament has at most three components for $k \ge 3$ by Corollary~\ref{cor:the number of components}.

We first characterize the niche-realizable pair $(G, k)$ for a graph $G$ with three components.

\begin{Thm}\label{thm:3components}
Let $G$ be a graph with three components and $k$ be an integer greater than or equal to $3$.
Then $(G, k)$ is niche-realizable if and only if $k=3$ and $G$ is isomorphic to $K_p \cup K_q \cup K_r$ for positive integers $p$, $q$, and $r$.
\end{Thm}
\begin{proof}
Suppose that $(G, k)$ is niche-realizable.
If there exists a component which is not isomorphic to a complete graph, then $\alpha(G)\geq 4$, which contradicts Theorem~\ref{thm:stability number}. Therefore $G$ is isomorphic to $K_p \cup K_q \cup K_r$ for positive integers $p$, $q$, and $r$.
Since $K_p \cup K_q \cup K_r$ is disconnected, $k \le 3$ by Theorem~\ref{thm:K4 connected}.
Therefore the ``only if'' part is true.

To show the ``if'' part, let $D$ be a digraph with the vertex set \[\{x_1, \ldots, x_p, y_1, \ldots, y_q,z_1, \ldots, z_r\}\] and the arc set $$\{(x_i, y_j) \mid i \in [p] \text{ and } j \in [q]\} \cup \{(y_j, z_l) \mid j \in [q] \text{ and } l \in [r]\} \cup \{(z_l, x_i) \mid l \in [r] \text{ and } i \in [p]\}.$$
Then it is easy to check that $D$ is a $3$-partite tournament and the niche graph of $D$ is isomorphic to $K_p \cup K_q \cup K_r$.
Hence the ``if'' part is true.
\end{proof}

Let $G$ be a graph.
Two vertices $u$ and $v$ of $G$ are said to be \emph{true twins} if they have the same closed neighborhood, and denoted by $u \equiv_G v$.
We may introduce an analogous notion for a digraph.
Let $D$ be a digraph.
Two vertices $u$ and $v$ of $D$ are said to be \emph{true twins} if they have the same open out-neighborhood and open in-neighborhood, and denoted by $u \equiv_D v$.

The following lemma is true by definitions of niche graph and true twins.
\begin{Lem}\label{lem:homo}
Let $D$ be a digraph without isolated vertices.
If vertices $u$ and $v$ are true twins in $D$, then $u$ and $v$ are true twins in $\NNN(D)$.
\end{Lem}
\begin{proof}
Suppose that two vertices $u$ and $v$ are true twins in $D$.
Then $N^+_D(u)=N^+_D(v)$ and $N^-_D(u)=N^-_D(v)$.
Therefore, by the definition of niche graph, $u$ and $v$ have the same open neighborhood in $\NNN(D)$.
Since $D$ has no isolated vertices, $N^+_D(u)\neq \emptyset$ or $N^-_D(u)\neq \emptyset$.
Thus $u$ and $v$ have a common out-neighbor or a common in-neighbor in $D$ and so they are adjacent in $\NNN(D)$.
Hence $u$ and $v$ have the same closed neighborhood in $\NNN(D)$.
\end{proof}

\begin{Lem} \label{lem:homo-same-partite}
Let $D$ be a multipartite tournament.
If vertices $u$ and $v$ are true twins in $D$, then $u$ and $v$ are in the same partite set.
\end{Lem}

\begin{proof}
Suppose that vertices $u$ and $v$ are true twins in $D$.
If $u$ and $v$ are not in the same partite set, then we may assume $(u,v) \in A(D)$ and so, by the definition of true twins, $(v,v) \in A(D)$, which contradicts the hypothesis that $D$ is a multipartite tournament.
\end{proof}

\begin{Prop} \label{prop:G-v-}
Given a graph $G$ with at least four vertices,
suppose that $G$ is niche-realizable through a $k$-partite tournament $D$ for $ k \geq 3$, and vertices $u$ and $v$ are true twins in $D$.
Then $D-v$ is a $k$-partite tournament whose niche graph is $G-v$.
\end{Prop}
\begin{proof}
Let $D'=D-v$.
By lemma~\ref{lem:homo-same-partite},
$D'$ is a $k$-partite tournament.
Since $D'$ is a subdigraph of $D$,
${\mathcal N}(D')$ is a subgraph of $G$ by Lemma~\ref{lem:subgraph}.
Therefore ${\mathcal N}(D')$ is a subgraph of $G-v$.
To show that $ G-v$ is a subgraph of ${\mathcal N}(D')$,
take an edge $xy$ in $G-v$.
Then $xy$ is an edge in $G$, so
$N^+_D(x)\cap N^+_D(y) \neq \emptyset$ or $N^-_D(x)\cap N^-_D(y) \neq \emptyset$.
By symmetry, we assume that $N^+_D(x)\cap N^+_D(y) \neq \emptyset$.
If $v \in N^+_D(x)\cap N^+_D(y)$, then
$u \in N^+_D(x)\cap N^+_D(y)$ and so
$u \in N^+_{D'}(x)\cap N^+_{D'}(y)$.
If $v \notin N^+_D(x)\cap N^+_D(y)$, then $N^+_D(x)\cap N^+_D(y)=N^+_{D'}(x)\cap N^+_{D'}(y)$.
Therefore we may conclude that $xy$ is an edge in ${\mathcal N}(D')$.
Thus $G-v$ is a subgraph of ${\mathcal N}(D')$ and so ${\mathcal N}(D') = G-v$.
\end{proof}

 \begin{Lem} \label{lem:1-1}
 Let $D$ be an orientation of $K_{2,1,1}$ with true twins. Then the niche graph of $D$ either is connected or has three components.
 \end{Lem}

 \begin{proof}
We denote the partite sets of $D$ by $(X_1,X_2,X_3)$.
Then we may assume that $X_1=\{x_1,x_2\}$, $X_2=\{x_3\}$, and $X_3=\{x_4\}$.
By the hypothesis, $D$ has true twins and so, by Lemma~\ref{lem:homo-same-partite}, $x_1$ and $x_2$ are true twins.
By Lemma~\ref{lem:inverse}, there are two cases to consider: $d_{D}^+(x_1)=2$; $d_{D}^+(x_1)=1$.
We first consider the case $d_{D}^+(x_1)=2$.
Then $N^+_D(x_1)=\{x_3,x_4\}$.
Since $x_1$ and $x_2$ are true twins,
$N^+_D(x_2)=\{x_3,x_4\}$.
Therefore
$N^-_D(x_3)\cap N^-_D(x_4) \neq \emptyset$.
Thus $x_3$ is adjacent to $x_4$ in ${\mathcal N}(D)$.
By symmetry, we may assume $N^+_D(x_3)=\{x_4\}$.
Then $N^-_D(x_4)=\{x_1,x_2,x_3\}$, so $\{x_1,x_2,x_3\}$ forms a clique in ${\mathcal N}(D)$.
Therefore ${\mathcal N}(D)$ is connected.

Now we consider the case $d_{D}^+(x_1)=1$.
Without loss of generality, we may assume that
$N^+_D(x_1)=\{x_3\}$. Then $N^+_D(x_2)=\{x_3\}$ and
$N^-_D(x_1)=N^-_D(x_2)=\{x_4\}$.
If $(x_3,x_4) \in A(D)$, then ${\mathcal N}(D) \cong K_2 \cup K_1 \cup K_1$.
Therefore ${\mathcal N}(D)$ has three components.
Suppose that $(x_3,x_4) \notin A(D)$, i.e. $(x_4,x_3) \in A(D)$.
Then $N^-_D(x_3)=\{x_1,x_2,x_4\}$, so $\{x_1,x_2,x_4\}$ forms a clique in  ${\mathcal N}(D)$.
Since $N^+_D(x_4)=\{x_1,x_2,x_3\}$, $\{x_1,x_2,x_3\}$ forms a clique in  ${\mathcal N}(D)$.
Therefore ${\mathcal N}(D)$ is connected.
 \end{proof}

\begin{Lem}\label{lem:4vertices diconnected}
Let $D$ be an orientation of $K_{2, 1, 1}$ whose niche graph is disconnected. Suppose that no two vertices are true twins in $D$.
Then the niche graph of $D$ is isomorphic to $P_3 \cup K_1$.
\end{Lem}
\begin{proof}
Let $\{x_1, x_2\}$, $\{x_3\}$, and $\{x_4\}$ be the partite sets of $D$. First we consider the case $|N^+_D(x_1)|=2$ or $|N^+_D(x_2)|=2$, i.e. $N_D^+(x_1)=\{x_3,x_4\}$ or $N_D^+(x_2)=\{x_3,x_4\}$.
By symmetry, we may assume that $N_D^+(x_1)=\{x_3,x_4\}$. Then $x_3$ and $x_4$ are adjacent in $\NNN(D)$.
Since $x_1$ and $x_2$ are not true twins in $D$, at least one of $x_3$ and $x_4$ is an in-neighbor of $x_2$. We may assume that $x_4$ is an in-neighbor of $x_2$. Suppose, to the contrary, that $x_1$ and $x_2$ are adjacent in $\NNN(D)$, then $x_3$ is a common out-neighbor of $x_1$ and $x_2$. If $(x_3,x_4) \in A(D)$ (resp.\ $(x_4,x_3) \in A(D)$), then $x_3$ (resp.\ $x_4$) is adjacent to $x_1$ in $\NNN(D)$. In either case, $\NNN(D)$ is connected and we reach a contradiction. Thus $x_1$ and $x_2$ are not adjacent in $\NNN(D)$ and so $N_D^-(x_2)=\{x_3,x_4\}$.

We denote $D_1$ the subdigraph of $D$ induced by $\{x_1, x_3, x_4\}$.
Since $N_D^+(x_1)=\{x_3,x_4\}$, $D_1$ is not a directed cycle.
Thus, by Lemma~\ref{lem:K3}, $\NNN(D_1)$ is connected, and so, by Lemma~\ref{lem:subgraph}, the subgraph of $\NNN(D)$ induced by $\{x_1, x_3, x_4\}$ is connected.
By applying a similar argument to the subdigraph induced by $\{x_2, x_3, x_4\}$, we may show that the subgraph of $\NNN(D)$ induced by $\{x_2, x_3, x_4\}$ is connected. Therefore $\NNN(D)$ is connected and  we reach a contradiction.
Thus $|N^+_D(x_1)|=2$ or $|N^+_D(x_2)|=2$ cannot happen. Then, by Lemma~\ref{lem:inverse}, the case $N_D^+(x_1) = \emptyset$ or $N_D^+(x_2) = \emptyset$ cannot happen.
Thus $|N_D^+(x_1)|=|N_D^+(x_2)|=1$.
If $N_D^+(x_1) = N_D^+(x_2)$, then $N_D^-(x_1) = N_D^-(x_2)$ and so $x_1$ and $x_2$ are true twins, which is a contradiction.
Therefore $N_D^+(x_1) \neq N_D^+(x_2)$.
Thus either $N_D^+(x_1)= \{x_3\}$ and $N_D^+(x_2)=\{x_4\}$ or $N_D^+(x_1)= \{x_4\}$ and $N_D^+(x_2)=\{x_3\}$.
By symmetry,  we may assume $N_D^+(x_1)= \{x_3\}$ and $N_D^+(x_2)=\{x_4\}$.
Since $x_3$ and $x_4$ belong to different partite sets in $D$, $(x_3, x_4) \in A(D)$ or $(x_4, x_3) \in A(D)$.
By symmetry again, we may assume that $(x_3, x_4) \in A(D)$.
Then $D$ is isomorphic to the digraph given in Figure~\ref{fig:K112 case}.
Hence the niche graph of $D$ is isomorphic to $P_3 \cup K_1$.
\end{proof}

\begin{figure}
\begin{center}

\begin{tikzpicture}[x=1.5cm, y=1.5cm]

    \vertex (x) at (0,0) [label=below:$x_1$]{};
    \vertex (y) at (0,1) [label=above:$x_2$]{};
    \vertex (z) at (1,0) [label=below:$x_3$]{};
    \vertex (w) at (1,1) [label=above:$x_4$]{};

    \path
    (x) edge [->,thick] (z)
    (z) edge [->,thick] (y)
    (y) edge [->,thick] (w)
    (w) edge [->,thick] (x)
    (z) edge [->,thick] (w)
	;
\end{tikzpicture}
\qquad \qquad \qquad
\begin{tikzpicture}[x=1.5cm, y=1.5cm]

    \vertex (x) at (0,0) [label=below:$x_1$]{};
    \vertex (y) at (0,1) [label=above:$x_2$]{};
    \vertex (z) at (1,0) [label=below:$x_3$]{};
    \vertex (w) at (1,1) [label=above:$x_4$]{};

    \path
    (y) edge [-,thick] (z)
    (y) edge [-,thick] (w)

	;
\end{tikzpicture}
\end{center}
\caption{An orientation of $K_{2,1,1}$ and its niche graph isomorphic to $ P_3 \cup K_1$}
\label{fig:K112 case}
\end{figure}
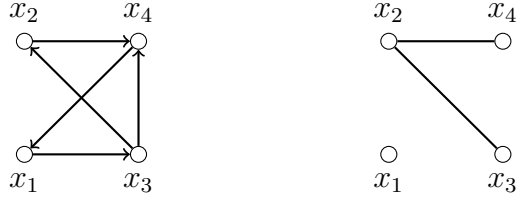

\begin{Lem}\label{lem:5vertices diconnected}
For positive integers $n_1, n_2$, and $n_3$ satisfying $n_1+n_2+n_3 \ge 5$,  suppose that an orientation $D$ of $K_{n_1, n_2, n_3}$ has no true twins.
Then the niche graph of $D$ is connected.
\end{Lem}
\begin{proof}
Without loss of generality, we may assume that $n_1 \ge n_2 \ge n_3$.
We first consider the case $n_1+n_2+n_3 = 5$. Then $n_1 =2$ or $3$. We will show that $\NNN(D)$ is connected in each of the following cases.

{\it Case} 1.  $n_1=2$.
Then $n_2=2$ and $n_3=1$.
Let $\{u_1, u_2\}$, $\{v_1, v_2\}$, and $\{w\}$ be the partite sets of $D$.
By Lemma~\ref{lem:inverse}, we may assume $d_D^+(w) \ge 2$.
Suppose $d_D^+(w) = 4$. Then $u_1$, $u_2$, $v_1$, and $v_2$ form a clique in $\NNN(D)$. If $(u_1,v_1) \in A(D)$ (resp.\ $(v_1,u_1)\in A(D)$),  then $v_1$ (resp.\ $u_1$) is a common out-neighbor of $u_1$ (resp.\ $v_1$) and $w$ and so $\NNN(D)$ is connected.

We consider the case $d_D^+(w) = 3$.
Then $N_D^+(w)=\{u_2,v_1,v_2\}$, $\{u_1,v_1,v_2\}$, $\{u_1,u_2,v_2\}$, or $\{u_1,u_2,v_1\}$.
By symmetry, we may assume that $N_D^+(w)=\{u_1, u_2, v_1\}$.
Then $N_D^-(w)=\{v_2\}$. Moreover, the subdigraphs $D_1$ and $D_2$ of $D$ induced by $\{w, u_1, v_1\}$ and by $\{w, u_2, v_1\}$, respectively, are orientations of $K_3$ which are not  directed cycles. Thus, by Lemma~\ref{lem:K3}, $\NNN(D_1)$ and $\NNN(D_2)$ are connected.
Since $D_1$ and $D_2$ are subdigraphs of $D$, by Lemma~\ref{lem:subgraph}, the subgraphs of $\NNN(D)$ induced by $\{w, u_1, v_1\}$ and by $\{w, u_2, v_1\}$ are connected respectively and so the subgraph of $\NNN(D)$ induced by $\{w,u_1,u_2,v_1\}$ is connected.
If $(v_2, u_1) \in A(D)$ or $(v_2, u_2) \in A(D)$, then $w$ and $v_2$ are adjacent in $\NNN(D)$  and we are done. Suppose that $(u_1, v_2) \in A(D)$ and $(u_2, v_2) \in A(D)$.
If $(v_1, u_1) \in A(D)$ and  $(v_1, u_2) \in A(D)$, then $N_D^+(u_1)=N_D^+(u_2)$ and $N_D^-(u_1)=N_D^-(u_2)$, which contradicts the hypothesis.
Therefore $(u_1, v_1) \in A(D)$ or $(u_2, v_1) \in A(D)$, and so $v_1$ is adjacent to $v_2$ in $\NNN(D)$.
Thus $\NNN(D)$ is connected.

We consider the case $d_D^+(w) = 2$.
Then one of the following is true:
\begin{itemize}	
\item $|N_D^+(w) \cap \{u_1, u_2\}|=1$ and $|N_D^+(w) \cap \{v_1, v_2\}|=1$;
\item $N_D^+(w) = \{u_1, u_2\}$ or $N_D^+(w) = \{v_1, v_2\}$.
\end{itemize}
We first suppose that  $|N_D^+(w) \cap \{u_1, u_2\}|=1$ and $|N_D^+(w) \cap \{v_1, v_2\}|=1$.
By symmetry, we may assume that $N_D^+(w)=\{u_1, v_1\}$.
Then $N_D^-(w)=\{u_2, v_2\}$.
Therefore the subdigraphs $D_3$ and $D_4$ of $D$ induced by $\{w, u_1, v_1\}$ and by $\{w, u_2, v_2\}$ are orientations of $K_3$ which are not directed cycles.
Then, by Lemma~\ref{lem:K3}, both $\NNN(D_3)$ and $\NNN(D_4)$ are connected.
Therefore, by Lemma~\ref{lem:subgraph}, the subgraphs of $\NNN(D)$ induced by $\{w, u_1, v_1\}$ and by $\{w, u_2, v_2\}$ are connected respectively.
Thus $\NNN(D)$ is connected.
Now suppose $N_D^+(w) = \{u_1, u_2\}$ or $N_D^+(w) = \{v_1, v_2\}$.
By symmetry, we may assume that $N_D^+(w) = \{u_1, u_2\}$.
Then $N_D^-(w) = \{v_1, v_2\}$.
Then $u_1$ and $u_2$ are adjacent and $v_1$ and $v_2$ are adjacent in $\NNN(D)$.
If $(u_j, v_i) \in A(D)$ for all $1 \le i, j \le 2$,
then $N_D^+(u_1) = N_D^+(u_2)$ and $N_D^-(u_1) = N_D^-(u_2)$, which contradicts the hypothesis.
Thus $(v_i, u_j) \in A(D)$ for some  $i$ and $j$ in $\{1,2\}$.
Then $u_j$ (resp.\ $v_i$) is a common out-neighbor (resp.\ common in-neighbor) of $v_i$ and $w$ (resp.\ $u_j$ and $w$) in $D$.
Thus each of $v_i$ and $u_j$ is adjacent to $w$ in $\NNN(D)$ and so $\NNN(D)$ is connected.
Hence we have shown that $\NNN(D)$ is connected if $n_1=2$.

{\it Case} 2. $n_1=3$.
Then $n_2=n_3=1$.
Let $\{x_1, x_2, x_3\}$, $\{y\}$, and $\{z\}$ be the partite sets of $D$.
We note that  $N_D^+(x_i) = N_D^+(x_j)$ if and only if $N_D^-(x_i) = N_D^-(x_j)$ for each $1 \le i < j \le 3$.
Therefore, by the hypothesis, $N_D^+(x_i) \neq N_D^+(x_j)$ for each $1 \le i < j \le 3$.
Then, since $N_D^+(x_i)$ is one of $\emptyset$, $\{y\}$, $\{z\}$, and $\{y, z\}$ for each $i=1$, $2$, and $3$, $d_D^+(x_i) = 1$ for some $i \in \{1,2,3\}$ and  $d_D^+(x_j) \neq 1$ for some $j \in \{1,2,3\}\setminus\{i\}$.
By symmetry, we may assume that $d_D^+(x_1)=1$ and $d_D^+(x_2) \in \{0, 2\}$.
In addition, by Lemma~\ref{lem:inverse}, we may assume that $d_D^+(x_2)=2$, i.e. $N_D^+(x_2)=\{y, z\}$.
Then $x_1$ and $x_2$ have a common out-neighbor in $D$, so $x_1$ and $x_2$ are adjacent in $\NNN(D)$.
On the other hand, since $y$ and $z$ belong to different partite sets, there is an arc between $y$ and $z$ and so the subdigraph $D_5$ of $D$ induced by $\{x_2, y, z\}$ is an orientation of $K_3$.
Since $N_D^+(x_2)=\{y, z\}$, $D_5$ is not a directed cycle, and so, by Lemma~\ref{lem:K3}, $\NNN(D_5)$ is connected.
Thus, by Lemma~\ref{lem:subgraph}, the subgraph of $\NNN(D)$ induced by $\{x_2, y, z\}$ is connected.
Since $x_1$ and $x_2$ are adjacent in $\NNN(D)$, the subgraph of $\NNN(D)$ induced by $\{x_1, x_2, y, z\}$ is connected.
We will show that $x_3$ is adjacent to a vertex in $\{x_1, x_2, y, z\}$ in $\NNN(D)$ to take care of this case.
If $x_3$ has an out-neighbor in $D$, then $x_2$ and $x_3$ are adjacent in $\NNN(D)$ and so we are done.
Suppose that $d_D^+(x_3)=0$.
Then the subdigraph of $D$ induced by $\{x_3,y,z\}$ is an orientation of $K_3$ which is not a directed cycle. By applying the same argument for $D_5$, we may show that  $\NNN(D)$ is connected.
Hence we have shown that $\NNN(D)$ is connected in the case $n_1+n_2+n_3=5$.

Now suppose that $n_1+n_2+n_3>5$.  To show that $\NNN(D)$ is connected, take two vertices $w_1$ and $w_2$ in $D$. Then we may take three vertices $w_3$, $w_4$, and $w_5$ in $D$ such that the induced subdigraph $D_6$ of $D$ induced by $\{w_1, w_2, w_3, w_4, w_5\}$ is a $3$-partite tournament.
By the above argument, $\NNN(D_6)$ is connected, so there is a $(w_1, w_2)$-path $P$ in $\NNN(D_6)$.
Since $D_6$ is a subdigraph of $D$, $\NNN(D_6)$ is a subgraph of $\NNN(D)$ by Lemma~\ref{lem:subgraph}.
Thus $P$ is a $(w_1, w_2)$-path in $\NNN(D)$ and hence $\NNN(D)$ is connected.
This completes the proof.
\end{proof}

For a graph $G$, a vertex $v$ of $G$, and a finite set $K$ disjoint from $V(G)$, we say that $v$ \emph{is replaced with a clique} formed by $K$ to obtain a new graph  with the vertex set $(V(G) \cup K) \setminus \{v\}$ and the edge set \[ E(G-v) \cup \{wx \mid w \neq x,  \{w,x\}\subset K\} \cup \{uw \mid uv \in E(G),  w\in K\}.\]
See Figure~\ref{fig:replace} for an illustration.
\begin{figure}
\begin{center}

\begin{tikzpicture}[x=1.5cm, y=1.5cm]

    \vertex (x1) at (0,0) [label=above:$$]{};
    \vertex (x2) at (0,1) [label=above:$$]{};
    \vertex (x3) at (1,0) [label=above:$$]{};
    \vertex (v) at (1,1) [label=above:$v$]{};
    \vertex (x5) at (1.6,1.6) [label=above:$$]{};

    \path
    (x1) edge [-,thick] (x2)
    (x1) edge [-,thick] (x3)
    (v) edge [-,thick] (x3)
    (v) edge [-,thick] (x2)
    (v) edge [-,thick] (x5)
	;

    \vertex (y1) at (5,0) [label=above:$$]{};
    \vertex (y2) at (5,1) [label=above:$$]{};
    \vertex (y3) at (6,0) [label=above:$$]{};
    \vertex (v1) at (6,1) [label=above:$$]{};
    \vertex (v2) at (6.2,0.8) [label=above:$$]{};
    \vertex (v3) at (5.8,1.2) [label=above:$$]{};
    \vertex (y5) at (6.6,1.6) [label=above:$$]{};

    \draw[dashed] (6,1) circle (0.4) [label=right:$K$]{};
    \draw (5.6,1.4) node{$K$};

    \path

    (v1) edge [-,thick] (v2)
    (v2) edge [-,bend right=40,thick] (v3)
    (v3) edge [-,thick] (v1)

    (y1) edge [-,thick] (y2)
    (y1) edge [-,thick] (y3)
    (v1) edge [-,thick] (y3)
    (v1) edge [-,thick] (y2)
    (v1) edge [-,thick] (y5)

    (y1) edge [-,thick] (y2)
    (y1) edge [-,thick] (y3)
    (v2) edge [-,thick] (y3)
    (v2) edge [-,thick] (y2)
    (v2) edge [-,thick] (y5)

    (y1) edge [-,thick] (y2)
    (y1) edge [-,thick] (y3)
    (v3) edge [-,thick] (y3)
    (v3) edge [-,thick] (y2)
    (v3) edge [-,thick] (y5)
	;

\end{tikzpicture}
\end{center}
\caption{The vertex $v$ of the graph on the left is replaced with a clique $K$ of size $3$ to yield the graph on the right.}
\label{fig:replace}
\end{figure}
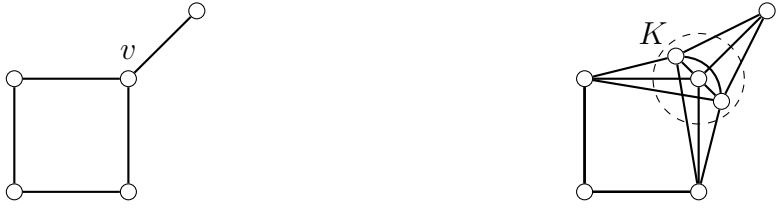
We call a graph an \emph{expansion} of a graph $G$ if it is obtained by replacing each vertex in $G$ with a clique (possibly of size $1$).

\begin{Thm}\label{thm:2components}
Let $G$ be a graph having exactly two components.
For $k \ge 3$, $(G, k)$ is niche-realizable if and only if $k=3$ and $G$ is isomorphic to an expansion of $P_3 \cup K_1$.
\end{Thm}
\begin{proof}
To show the ``if" part,
suppose that $G$ is isomorphic to an expansion of $P_3 \cup K_1$.
We will show that $(G,3)$ is niche-realizable.
Let $D$ be the digraph given in Figure~\ref{fig:K112 case}.
Then $\NNN(D)$ is isomorphic to $P_3 \cup K_1$.
Let $X_i$ be the set of vertices of $G$ which are true twins to the vertex corresponding to $x_i$ in $\NNN(D)$ for each $1\leq i\leq 4$.
We construct a digraph $D^*$ from $D$ in the following way:
\[
V(D^*)=V(G);
\]
\[
A(D^*)=\{(v,w) \mid v \in X_i, w \in X_j, (i,j) \in \{ (1,3),(2,4),(3,2),(3,3),(4,1)\} \}.
\]
Then $D^*$ is a $3$-partite tournament, and
\begin{itemize}
  \item $N_{D^*}^+(u_1)=X_3$, $N_{D^*}^-(u_1)=X_4$;
  \item $N_{D^*}^+(u_2)=X_4$, $N_{D^*}^-(u_2)=X_3$;
  \item $N_{D^*}^+(u_3)=X_2 \cup X_4$, $N_{D^*}^-(u_3)=X_1$;
  \item $N_{D^*}^+(u_4)=X_1$, $N_{D^*}^-(u_4)=X_2 \cup X_3$
\end{itemize}
for each vertex $u_i \in X_i$; for each $1\leq i \leq 4$.
Thus $X_i$ forms a clique in $\NNN(D^*)$ for each $1 \leq i \leq 4$.
Take $v$ and $w$ in $G$.
We first consider the case in which $v$ and $w$ are adjacent in $G$.
Then $v$ and $w$ belong to $X_i$ for some $i \in \{1,2,3,4\}$ or exactly one of $v$ and $w$ belongs to $X_2$ and the other one belongs to $X_3 \cup X_4$.
If the former is true, then $v$ and $w$ are adjacent in $\NNN(D^*)$ by above observation.
Suppose the latter.
Then, without loss of generality, we may assume that $v$ belongs to $X_2$ and $w$ belongs to $X_3 \cup X_4$.
If $w$ belongs to $X_3$ (resp.\ $X_4$), then $v$ and $w$ have a common out-neighbor (resp.\ common in-neighbor) in $D^*$ by the above observation, and so they are adjacent in $\NNN(D^*)$.

Now we consider the case where $v$ and $w$ are not adjacent in $G$.
Then, without loss of generality, we may assume that $v$ belongs to $X_1$ and $w$ does not belong to $X_1$ or $v$ and $w$ belong to $X_3$ and $X_4$, respectively.
If the former is true, $N_{D^*}^+(v) = X_3$, $N_{D^*}^-(v)=X_4$, $N_{D^*}^+(w) \subset X_1 \cup X_2 \cup X_4$, and $N_{D^*}^-(w) \subset X_1 \cup X_2 \cup X_3$ by the above observation, and so $v$ and $w$ are not adjacent in $\NNN(D^*)$.
If the latter is true, $N_{D^*}^+(v) = X_2 \cup X_4$, $N_{D^*}^-(v)=X_1$, $N_{D^*}^+(w) = X_1$, and $N_{D^*}^-(w) = X_2 \cup X_3$ by the above observation, and so $v$ and $w$ are not adjacent in $\NNN(D^*)$.
Hence we have shown that $G$ is isomorphic to $\NNN(D^*)$.

To show the ``only if" part,
suppose that $(G,k)$ is a niche-realizable.
Let $D$ be a $k$-partite tournament whose niche graph is $G$.
Since $G$ is not connected, $k < 4$ by Theorem~\ref{thm:K4 connected} and so $k =3$.
Thus $D$ is an orientation of $K_{n_1,n_2,n_3}$ for positive integers $n_1$, $n_2$, and $n_3$.
If $|V(G)|=3$, then $D$ is an orientation of $K_3$ and so, by Lemma~\ref{lem:K3}, $G$ is connected  or has three components, which contradicts the hypothesis that $G$ has exactly two components.
Therefore $|V(G)| \geq 4$.
In the following, we show that $G$ is isomorphic to an expansion of $P_3 \cup K_1$ by induction on $|V(G)|$.
First we consider the case where $|V(G)|=4$.
Then $D$ is an orientation of $K_{2,1,1}$.
If $D$ has true twins, then $G$ is connected or has three components by Lemma~\ref{lem:1-1}, which is a contradiction.
Therefore $D$ has no true twins, so $G \cong P_3 \cup K_1$ by Lemma~\ref{lem:4vertices diconnected}.
Thus the basis step is true.

We assume that the statement is true for any niche-realizable graph on $l$ vertices which has exactly two components for a positive integer $l \geq 4$.
Now we assume $|V(G)|=l+1$.
Then $n_1+n_2+n_3=l+1 \geq 5$.
Since $G$ is not connected, $D$ has true twins by Lemma~\ref{lem:5vertices diconnected}.
Let $u$ and $v$ be true twins
in $D$.
Then $D-v$ is a $3$-partite tournament and $G-v$ is the niche graph of $D-v$ by Proposition~\ref{prop:G-v-}.
On the other hand, $u$ and $v$ are true twins in $G$ by Lemma~\ref{lem:homo}.
Then, $G$, $G-u$, and $G-v$ have the same number of components.
Since $G$ has two components by the hypothesis, $G-v$ has exactly two components.
Therefore, by the induction hypothesis,
$G-v$ is an expansion of $P_3 \cup K_1$.
Since $v$ and $u$ are true twins in $G$,
$G$ is an expansion of $P_3 \cup K_1$.
\end{proof}

\section{Niche-realizable pairs $(G,k)$ when $G$ is connected}
In this section, we study the niche graphs of $k$-partite tournaments for $k \ge 3$ which are connected.
We first find all the niche-realizable pairs $(K_n, k)$ for positive integers $n \ge k \ge 3$.

\begin{Thm}\label{thm:complete}
For positive integers $n \ge k \ge 3$, $(K_n, k)$ is niche-realizable if and only if $(n, k) \in \{(4, 4)\} \cup \{(n, k) \mid n \ge 5\}$.
\end{Thm}
\begin{proof}
To show the ``if'' part, we construct a digraph $D$ in the following way.
Let $V(D)= \{v_1, v_2, \ldots, v_n\}$.
If $k=3$ and $n \ge 5$, then let $D$
be any $3$-partite tournament with partite sets $\{v_1\}$, $\{v_2, v_3\}$, and $\{v_4, v_5, \ldots, v_n\}$ whose arc set includes the following arc set (the remaining arcs have an arbitrary orientation):
\begin{align*}
 \{(v_1, v_i) \mid 2 \le i \le n\} \cup \{(v_2, v_4), (v_4, v_3), (v_3, v_5), (v_5, v_2) \}
     \cup \{(v_i, v_2) \mid 6 \le i \le n \}.
\end{align*}
If $k \ge 4$ and $n \ge 4$, then let $D$ be any  $k$-partite tournament with partite sets $\{v_1\}$, $\{v_2\}$, $\ldots$, $\{v_{k-1}\}$, $\{v_k, v_{k+1}, \ldots, v_n\}$ whose arc set includes the following arc set (the remaining arcs have an arbitrary orientation):
\[
\{(v_1, v_i) \mid 2 \le i \le n\} \cup
\bigcup_{i=2}^{k-2} \left\{(v_i, v_{i+1})\right \}  \cup
\bigcup_{i=k}^{n} \left\{(v_{k-1}, v_i), (v_i, v_2) \right\}.
\]
In both cases, $v_1$ is a common in-neighbor of the remaining vertices, so $\{v_2, v_3, \ldots, v_n\}$ forms a clique in $\NNN(D)$.
Moreover, since $v_i$ has at least one out-neighbor in $\{v_2, v_3, \ldots, v_n\}$ for each $2 \leq i \leq n$, $v_1$ and $v_i$ have a common out-neighbor in $D$, and so they are adjacent in $\NNN(D)$.
Therefore $\NNN(D)$ is a complete graph with $n$ vertices.

Now we show the ``only if'' part.
By Lemma~\ref{lem:K3}, $(K_3, 3)$ is not niche-realizable.
We only need to show that $(K_4, 3)$ is not niche-realizable.
Suppose, to the contrary, that $(K_4, 3)$ is niche-realizable.
Then there is an orientation $D$ of $K_{1,1,2}$ such that $\NNN(D)$ is isomorphic to $K_4$.
Let $\{x\}$, $\{y\}$, and $\{z, w\}$ be the partite sets of $D$.
Since $\NNN(D) \cong K_4$, $z$ and $w$ are adjacent in $\NNN(D)$, and so have a common out-neighbor or a common in-neighbor in $D$.
By Lemma~\ref{lem:inverse}, we may assume that they have a common out-neighbor and, by symmetry, we may assume that $y$ is a common out-neighbor of $z$ and $w$.
Then, since $x$ and $z$ are adjacent in $\NNN(D)$, $(x, y) \in A(D)$.
Thus $N_D^-(y)=\{x,z,w\}$.
\begin{figure}
\begin{center}

\begin{tikzpicture}[x=1.8cm, y=1.8cm]

    \vertex (x) at (0,1) [label=above:$x$]{};
    \vertex (y) at (1,1) [label=above:$y$]{};
    \vertex (z) at (0,0) [label=below:$z$]{};
    \vertex (w) at (1,0) [label=below:$w$]{};
    \draw[dashed] (0,1.1) circle (0.3);
    \draw[dashed] (1,1.1) circle (0.3);
    \draw[dashed] (0.5,-0.1) ellipse (1.4cm and 0.5cm);

    \path
    (x) edge [->,thick] (y)
    (z) edge [->,thick] (y)
    (w) edge [->,thick] (y)
	;
\end{tikzpicture}
\end{center}
\caption{A subdigraph of $D$}
\label{fig:complete proof}
\end{figure}
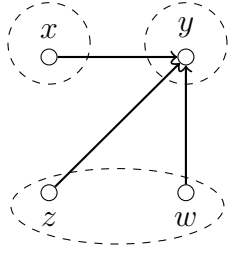
On the other hand, since $y$ and $z$ (resp.\ $w$) are adjacent in $\NNN(D)$, they have a common out-neighbor or a common in-neighbor in $D$.
Yet, $y$ has no out-neighbor in $D$, so $y$ and $z$ (resp.\ $w$) have a common in-neighbor that must be $x$ (see Figure~\ref{fig:complete proof}).
Then $A(D)=\{(x, y), (x, z), (x, w), (z, y), (w, y)\}$.
Since $x$ has only out-neighbors and $y$ has only in-neighbors, they are not adjacent in $\NNN(D)$, which is a contradiction to the supposition that $\NNN(D) \cong K_4$.
Hence the ``only if'' part is true.
\end{proof}

The rest of this paper will be devoted to
finding all the niche-realizable pairs $(G,k)$ when $G$ is connected triangle-free.

\begin{Lem}\label{lem:equivalence contains a triangle}
Let $D$ be a digraph with at least three vertices whose niche graph $\NNN(D)$ is connected.
If there are two distinct vertices which are true twins in $D$, then $\NNN(D)$ contains a triangle.
\end{Lem}
\begin{proof}
Suppose that $u$ and $v$ are distinct vertices which are true twins in $D$.
Since $\NNN(D)$ is connected and has at least three vertices, $D$ contains a vertex $w$ other than $u$ and $v$ that is adjacent to $u$ or $v$ in $\NNN(D)$.
Without loss of generality, we may assume that $w$ is adjacent to $v$ in $\NNN(D)$.
Since $\NNN(D)$ is connected, $D$ has no isolated vertices.
Then $u$ and $v$ are true twins in $\NNN(D)$ by Lemma~\ref{lem:homo}.
Thus $\{u, v, w\}$ forms a triangle in $\NNN(D)$.
\end{proof}

We make the following rather obvious observation.

\begin{Lem}\label{lem:in and out}
Let $D$ be a $k$-partite tournament for $k \ge 3$.
Then, for each partite set $X$ and each $x \in X$, $N_D^+(x) \cup N_D^-(x) = V(D) \setminus X$.
\end{Lem}

\begin{Thm}\label{thm:diameter}
Let $D$ be a $k$-partite tournament for $k \ge 3$.
Then $\NNN(D)$ contains no induced path of length $5$, that is, $\NNN(D)$ is $P_6$-free.
\end{Thm}
\begin{proof}
We denote the partite sets of $D$ by $X_1$, $\ldots$, $X_{k-1}$, and $X_k$.
If $\NNN(D)$ is disconnected, it contains no induced path of length $5$ by Corollary~\ref{cor:the number of components} and Theorems~\ref{thm:3components} and~\ref{thm:2components}.

Suppose that $\NNN(D)$ is connected.
To reach a contradiction, suppose that $\NNN(D)$ contains an induced path $P$ of length $5$.
Let $P=x_1x_2x_3x_4x_5x_6$.
Suppose that $|X_i \cap V(P)| \le 1$ for some $i \in [k]$.
Without loss of generality, we may assume that $|X_1 \cap V(P)| \le 1$.
Take a vertex $x \in X_1$.
Then $N_D^+(x) \cup N_D^-(x)$ contains at least five vertices in $V(P)$ by Lemma~\ref{lem:in and out}.
Therefore $N_D^+(x)$ or $N_D^-(x)$ contains at least three vertices in $V(P)$.
Since each of $N_D^+(x)$ and $N_D^-(x)$ forms a clique in $\NNN(D)$, the subgraph of $\NNN(D)$ induced by $V(P)$ contains a triangle, which contradicts the choice of $P$ as an induced path of $\NNN(D)$.
Thus $|X_i \cap V(P)| \ge 2$ for each $1 \leq i \leq k$.
Since $|V(P)|=6$, $k \ge 3$, and $X_1$, $\ldots$, $X_k$ are mutually disjoint, we obtain $k=3$ and
\begin{equation}\label{eqn:X_i cap V(P)}
|X_i \cap V(P)|=2
\end{equation}
for each $i=1$, $2$, and $3$.
Now let $D_1$ be the subdigraph of $D$ induced by $V(P)$.
Then $D_1$ is a $3$-partite tournament.
By Lemma~\ref{lem:subgraph}, $\NNN(D_1)$ is a subgraph of $P$.
Thus $\NNN(D_1)$ is triangle-free and so, by Lemma~\ref{lem:in-out-degree-condition}, $d_{D_1}^+(x) \le 2$ and $d_{D_1}^-(x) \le 2$ for all $x \in V(D_1)$.
By \eqref{eqn:X_i cap V(P)}, $d_{D_1}^+(x)+d_{D_1}^-(x)=4$,  so
\begin{equation}\label{eqn:d(x_l)=2}
d_{D_1}^+(x) = 2  \quad \text{and} \quad d_{D_1}^-(x) =2
\end{equation}
for all $x \in V(D_1)$.

Suppose that $\NNN(D_1)$ is disconnected.
Then $x_j$ and $x_{j+1}$ are not adjacent in $\NNN(D_1)$ for some $j \in \{1, 2, 3, 4, 5\}$, so
\begin{equation}\label{eqn:N(D_1)}
N_{D_1}^+(x) \neq \{x_j, x_{j+1}\} \quad \mbox{and} \quad  N_{D_1}^-(x) \neq \{x_j, x_{j+1}\}
\end{equation}
for all $x \in V(D_1)$.
Yet, since $x_j$ and $x_{j+1}$ are adjacent in $\NNN(D)$, they have a common in-neighbor or a common out-neighbor in $D$.
By Lemma~\ref{lem:inverse}, we may assume that $x_j$ and $x_{j+1}$ have a common out-neighbor $y$ in $D$.
Obviously $y \notin V(D_1)$.
Without loss of generality, we may assume that $y \in X_1$.
Then $x_j$ and $x_{j+1}$ do not belong to $X_1$.
By \eqref{eqn:X_i cap V(P)}, $|V(P) \setminus X_1|=4$, so $|(N_{D}^+(y) \cup N_{D}^-(y))\cap V(P)|=4$ by Lemma~\ref{lem:in and out}.
Since $P$ is an induced path of $D$,
$|N_{D}^-(y)\cap V(P)| = 2$ and $|N_{D}^+(y)\cap V(P)| = 2$.
Thus $N_{D}^-(y)\cap V(P) =\{x_j, x_{j+1}\}$.
Since $|N_{D}^+(y) \cap V(P)| = 2$, $N_{D}^+(y)\cap V(P) $ also forms an edge in $\NNN(D)$, that is, $N_{D}^+(y) \cap V(P) =\{x_k, x_{k+1}\}$ for some $k \in \{1, 2, 3, 4, 5\} \setminus \{j-1, j, j+1\}$.
Therefore $V(P) \setminus X_1 = \{x_j, x_{j+1}, x_k, x_{k+1} \}$.
Let $z$ be one of the two vertices in $X_1 \cap V(D_1)$.
Then $z \neq y$.
By Lemma~\ref{lem:in and out}, $N_{D_1}^+(z) \cup N_{D_1}^-(z) = \{x_j, x_{j+1}, x_k, x_{k+1}\}$.
By \eqref{eqn:d(x_l)=2}, $d_{D_1}^+(z)=d_{D_1}^-(z)=2$.
Then, by \eqref{eqn:N(D_1)}, $$\{N_{D_1}^+(z), N_{D_1}^-(z)\}=\{\{x_j, x_k\}, \{x_{j+1}, x_{k+1}\}\} \text{ or } \{N_{D_1}^+(z), N_{D_1}^-(z)\}=\{\{x_j, x_{k+1}\}, \{x_{j+1}, x_{k}\}\}.$$
In the former case, $x_j$ and $x_k$ are adjacent in $\NNN(D_1)$ and so in $\NNN(D)$, which is impossible as $P$ is an induced path in $\NNN(D)$.
In the latter case, $x_j$ and $x_{k+1}$ are adjacent and $x_{j+1}$ and $x_k$ are adjacent in $\NNN(D)$.
However, either $x_j$ and $x_{k+1}$ or $x_{j+1}$ and $x_k$ are not consecutive on $P$ and we reach a contradiction.
Thus $\NNN(D_1)$ is connected.
Since $P$ is an induced path of $\NNN(D)$ and $\NNN(D_1)$ is a spanning subgraph of $P$, we may conclude that $\NNN(D_1)=P$.

Let $D_2=D_1-x_2$.
Then $D_2$ is a $3$-partite tournament by \eqref{eqn:X_i cap V(P)} and, by Lemma~\ref{lem:subgraph}, $\NNN(D_2)$ is a subgraph of $\NNN(D_1)=P$.
Since $P - x_2$ is disconnected, $\NNN(D_2)$ is disconnected.
Without loss of generality, we may assume that $x_2 \in X_1$.
Then, by \eqref{eqn:X_i cap V(P)},
\begin{equation}\label{eqn:for D2}
|V(D_2) \cap X_1|=1 \quad \text{and} \quad |V(D_2) \cap X_2|=|V(D_2) \cap X_3|=2.
\end{equation}

Suppose that $u$ and $v$ are true twins in $D_2$ for some distinct vertices $u$ and $v$ in $V(D_2)$, that is, $N_{D_2}^+(u) = N_{D_2}^+(v)$ and $N_{D_2}^-(u) = N_{D_2}^-(v)$.
Then both $u$ and $v$ belong to the same partite set by Lemma~\ref{lem:homo-same-partite}.
Thus, by \eqref{eqn:for D2}, $u$ and $v$ belong to $X_2$ or $X_3$.
By \eqref{eqn:d(x_l)=2}, either $d_{D_2}^+(u)=d_{D_2}^+(v)=2$ and $d_{D_2}^-(u)=d_{D_2}^-(v)=1$ or
$d_{D_2}^+(u)=d_{D_2}^+(v)=1$ and $d_{D_2}^-(u)=d_{D_2}^-(v)=2$.
By Lemma~\ref{lem:inverse}, we may assume that $d_{D_2}^+(u)=d_{D_2}^+(v)=2$ and $d_{D_2}^-(u)=d_{D_2}^-(v)=1$.
Then $x_2$ is a common in-neighbor of $u$ and $v$ in $D_1$ by \eqref{eqn:d(x_l)=2}.
Thus $N_{D_1}^+(u) = N_{D_1}^+(v)$ and $N_{D_1}^-(u) = N_{D_1}^-(v)$, that is, $u$ and $v$ are true twins in $D_1$.
Since $|V(D_1)| \ge 3$ and $\NNN(D_1)$ is connected, $\NNN(D_1)$ contains a triangle  by Lemma~\ref{lem:equivalence contains a triangle}.
Yet, $\NNN(D_1)=P$ and we reach a contradiction.
Therefore there is no pair of vertices which are true twins in $D_2$.
Thus, by Lemma~\ref{lem:5vertices diconnected}, $\NNN(D_2)$ is connected and we reach a contradiction.
Hence $\NNN(D)$ contains no induced path of length $5$ and we are done.
\end{proof}

\noindent
From the above theorem, the following corollary immediately follows.

\begin{Cor} \label{cor:diameter}
Let $D$ be a $k$-partite tournament for $k \ge 3$.
Then each component of $\NNN(D)$ has diameter at most $4$.
\end{Cor}

A graph is said to be {\it triangle extended complete bipartite} if it is obtained from a complete bipartite graph by possibly attaching some $P_3$s to a common edge of the bipartite graph.
A set $U \subseteq V$
{\it dominates} a set $U' \subseteq V$ if any vertex $v \in  U'$ either lies in $U$ or has a neighbor in $U$. We also say that $U$ dominates $G[U']$. A
subgraph $H$ of $G$ is a {\it dominating subgraph} of $G$ if $V(H)$ dominates $G$.

Liu \emph{et al.}~\cite{van2010new} showed that a graph $G$ is $P_6$-free if and only if each connected induced subgraph of $G$ has a dominating (not necessarily
induced) triangle extended complete bipartite graph or an induced dominating $C_6$. Thus the following result immediately follows.

\begin{Cor}
Let $D$ be a $k$-partite tournament for $k \ge 3$. Then each connected induced subgraph of $\NNN(D)$ has a dominating (not necessarily
induced) triangle extended complete bipartite or an induced dominating $C_6$.
\end{Cor}

By using Theorem~\ref{thm:diameter}, we may find all the niche-realizable pairs $(P_n,k)$ and all the niche-realizable pairs $(C_n,k)$ for positive integers $n \geq k \geq 3$.

\begin{Lem}\label{lem:path}
For positive integers $n \ge k \ge 3$, $(P_n, k)$ is niche-realizable if and only if $(n, k) \in \{(3, 3), (4, 3), (4, 4), (5, 3)\}$.
\end{Lem}
\begin{proof}
Let $D_1$, $D_2$, $D_3$, and $D_4$ be the digraphs in Figure~\ref{fig:path} which are isomorphic to some orientations of $K_{1,1,1}$, $K_{1,1,2}$, $K_{1,1,1,1}$, and $K_{1,2,2}$, respectively.
It is easy to check that $\NNN(D_1) \cong P_3$, $\NNN(D_2)  \cong P_4$, $\NNN(D_3)  \cong P_4$, and $\NNN(D_4)  \cong P_5$.
Hence the ``if'' part is true.

Now suppose that $(P_n, k)$ is niche-realizable.
By Theorem~\ref{thm:diameter}, $n \le 5$.
Thus we only need to show that $(n, k)$ is neither $(5, 4)$ nor $(5, 5)$.
Let $D$ be a $k$-partite tournament such that $\NNN(D) \cong P_5$.
We denote $P_5$ by $x_1x_2x_3x_4x_5$.
Since $\NNN(D) \cong P_5$, $\NNN(D)$ is triangle-free and so, by Lemma~\ref{lem:in-out-degree-condition},  every vertex of $D$ has indegree at most two and outdegree at most two in $D$.
Suppose that $\{x_2\}$ is one of the partite sets of $D$.
Then $N_D^+(x_2) \cup N_D^-(x_2) = V(D) \setminus \{x_2\}$ by Lemma~\ref{lem:in and out}, so $d_D^+(x_2)=2$ and $d_D^-(x_2)=2$.
By Lemma~\ref{lem:inverse}, we may assume that $x_1$ is a out-neighbor of $x_2$ in $D$.
Since $N_D^+(x_2)$ forms an edge in $\NNN(D)$, $x_1$ is adjacent to a vertex in $P_5$ other than $x_2$ and we reach a contradiction.
Therefore $\{x_2\}$ is properly contained in a partite set of $D$.
Thus $k \neq 5$.
By symmetry, $\{x_4\}$ is properly contained in a partite set of $D$.
Now suppose that $k=4$.
Then $\{x_1\}$, $\{x_3\}$, $\{x_5\}$, and $\{x_2, x_4\}$ are the partite sets of $D$.
Therefore $d_D^+(x_2)+d_D^-(x_2) =3$ by Lemma~\ref{lem:in and out} and so $d_D^+(x_2)=2$ or $d_D^-(x_2)=2$.
By Lemma~\ref{lem:inverse}, we may assume that $d_D^+(x_2)=2$.
Then the out-neighbors of $x_2$ in $D$ are adjacent in $\NNN(D)$.
However, the possible out-neighbors of $x_2$ in $D$ are $x_1$, $x_3$, $x_5$ no two of which are consecutive on $P_5$.
Hence we have reached a contradiction and so $k=3$.
This completes the proof.
\end{proof}

\begin{figure}
\begin{center}

\begin{tikzpicture}[x=1.0cm, y=1.0cm]

    \vertex (x1) at (0,0) [label=above:$$]{};
    \vertex (x2) at (1,0) [label=above:$$]{};
    \vertex (x3) at (2,0) [label=above:$$]{};

    \path
    (x1) edge [->,thick] (x2)
    (x1) edge [->,bend left=25,thick] (x3)
    (x2) edge [->,thick] (x3)

	;
\draw (1, -0.5) node{$D_1$};
\end{tikzpicture}
\qquad \qquad\qquad
\begin{tikzpicture}[x=1.0cm, y=1.0cm]

    \vertex (x1) at (0,0) [label=above:$$]{};
    \vertex (x2) at (1,0) [label=above:$$]{};
    \vertex (x3) at (2,0) [label=above:$$]{};

    \path
    (x1) edge [-,thick] (x2)
    (x2) edge [-,thick] (x3)

	;
\draw (1, -0.5) node{$\NNN(D_1)$};

\end{tikzpicture}
\vskip1cm
\quad
\begin{tikzpicture}[x=1.0cm, y=1.0cm]

   \vertex (x1) at (0,0) [label=above:$$]{};
   \vertex (x2) at (1,0) [label=above:$$]{};
   \vertex (x3) at (1,1) [label=above:$$]{};
   \vertex (x4) at (0,1) [label=above:$$]{};
   \path
   (x2) edge [->,thick] (x1)
   (x3) edge [->,thick] (x1)
   (x3) edge [->,thick] (x2)
   (x4) edge [->,thick] (x2)
   (x4) edge [->,thick] (x3)
	;
\draw (0.5, -0.5) node{$D_2$};
\end{tikzpicture}
\qquad \qquad \qquad \qquad
\begin{tikzpicture}[x=1.0cm, y=1.0cm]

   \vertex (x1) at (0,0) [label=above:$$]{};
   \vertex (x2) at (1,0) [label=above:$$]{};
   \vertex (x3) at (1,1) [label=above:$$]{};
   \vertex (x4) at (0,1) [label=above:$$]{};
   \path
   (x1) edge [-,thick] (x2)
   (x2) edge [-,thick] (x3)
   (x3) edge [-,thick] (x4)

	;
\draw (0.5, -0.5) node{$\NNN(D_2)$};

\end{tikzpicture}

\vskip1cm
\quad
\begin{tikzpicture}[x=1.0cm, y=1.0cm]

   \vertex (x1) at (0,0) [label=above:$$]{};
   \vertex (x2) at (1,0) [label=above:$$]{};
   \vertex (x3) at (1,1) [label=above:$$]{};
   \vertex (x4) at (0,1) [label=above:$$]{};
   \path
   (x2) edge [->,thick] (x1)
   (x3) edge [->,thick] (x1)
   (x3) edge [->,thick] (x2)
   (x4) edge [->,thick] (x2)
   (x4) edge [->,thick] (x3)
   (x1) edge [->,thick] (x4)
	;
\draw (0.5, -0.5) node{$D_3$};
\end{tikzpicture}
\qquad \qquad \qquad \qquad
\begin{tikzpicture}[x=1.0cm, y=1.0cm]

   \vertex (x1) at (0,0) [label=above:$$]{};
   \vertex (x2) at (1,0) [label=above:$$]{};
   \vertex (x3) at (1,1) [label=above:$$]{};
   \vertex (x4) at (0,1) [label=above:$$]{};
   \path
   (x1) edge [-,thick] (x2)
   (x2) edge [-,thick] (x3)
   (x3) edge [-,thick] (x4)

	;
\draw (0.5, -0.5) node{$\NNN(D_3)$};

\end{tikzpicture}

\vskip1cm
\quad
\begin{tikzpicture}[x=1.0cm, y=1.0cm]

   \vertex (x1) at (0,0) [label=above:$$]{};
   \vertex (x2) at (1.25,0) [label=above:$$]{};
   \vertex (x3) at (0.625,0.625) [label=above:$$]{};
   \vertex (x4) at (0,1.25) [label=above:$$]{};
   \vertex (x5) at (1.25,1.25) [label=above:$$]{};
   \path
   (x2) edge [->,thick] (x1)
   (x1) edge [->,thick] (x3)
   (x1) edge [->,bend left=15,thick] (x5)

   (x2) edge [->,bend right=20,thick] (x3)
   (x4) edge [->,bend right=17,thick] (x2)

   (x3) edge [->,bend right=10,thick] (x4)
   (x3) edge [->,bend right=20,thick] (x5)

   (x5) edge [->,thick] (x4)
	;
	;
\draw (0.625, -0.5) node{$D_4$};
\end{tikzpicture}
\qquad  \qquad  \qquad \qquad
\begin{tikzpicture}[x=1.0cm, y=1.0cm]

   \vertex (x1) at (0,0) [label=above:$$]{};
   \vertex (x2) at (1.25,0) [label=above:$$]{};
   \vertex (x3) at (0.625,0.625) [label=above:$$]{};
   \vertex (x4) at (0,1.25) [label=above:$$]{};
   \vertex (x5) at (1.25,1.25) [label=above:$$]{};
   \path
   (x4) edge [-,thick] (x5)
   (x5) edge [-,thick] (x3)
   (x3) edge [-,thick] (x1)
   (x1) edge [-,thick] (x2)

	;
\draw (0.625, -0.5) node{$\NNN(D_4)$};
\end{tikzpicture}

\end{center}
\caption{The digraphs $D_1$, $D_2$, $D_3$, and $D_4$ which are isomorphic to some orientations of $K_{1,1,1}$, $K_{1,1,2}$, $K_{1,1,1,1}$, and $K_{1,2,2}$, respectively, and their niche graphs}
\label{fig:path}
\end{figure}
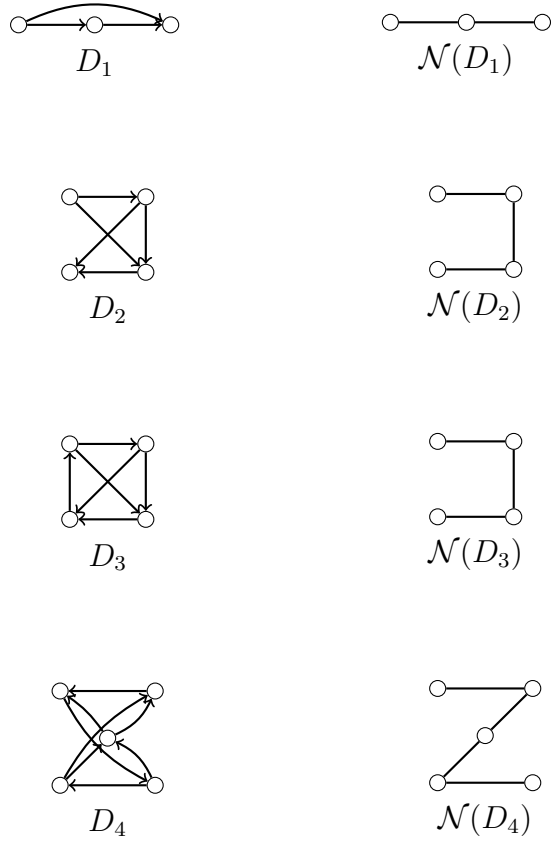

\begin{Lem} \label{prop:triangle-free-character}
For a $k$-partite tournament $D$ with $n$ vertices for some integers $n \geq k \geq 3$,
suppose that ${\mathcal N}(D)$ is a connected triangle-free graph.
Then
 $k \in \{ 3,4,5\}$ and
    \begin{equation} \label{eq:triangle-free-character}
\begin{cases} 3 \leq n \leq 6 & \mbox{if $k=3$} \\
4 \leq n \leq 5 & \mbox{if $k=4$} \\
  n=5 & \mbox{if $k=5$}. \end{cases}
\end{equation}
\end{Lem}

\begin{proof}
If $k \geq 6$, then $ 5 \leq d^+_D(v)+d^-_D(v)$ for each vertex $v$ in $D$ by Lemma~\ref{lem:in and out}, which contradicts Lemma~\ref{lem:in-out-degree-condition}.
Thus $k \leq 5$.
Let $X_i$ be a partite set of $D$ for each $1\leq i \leq k$.
Without loss of generality, we may assume that $X_1$ is a partite set with the smallest size among the partite sets.
Then $ |X_1| \leq \left\lfloor \frac{n}{k}\right \rfloor $.
Take a vertex $u$ in $X_1$.
By Lemma~\ref{lem:in and out},
$n-|X_1|=d^+_D(u)+d^-_D(u)$.
Since $d^+_D(u)+d^-_D(u) \leq 4$ by Lemma~\ref{lem:in-out-degree-condition}, $n-|X_1|\leq 4$ and so
\[n - \left\lfloor \frac{n}{k} \right\rfloor \leq 4.\]
It is easy to check that \eqref{eq:triangle-free-character} is an immediate consequence of this inequality.
\end{proof}

\begin{Lem}\label{lem:cycle}
For positive integers $n \ge k \ge 3$, $(C_n, k)$ is niche-realizable if and only if $(n, k) \in \{(5, 3), (5, 4), (5, 5), (6, 3)\}$.
\end{Lem}
\begin{proof}
Let $D_1$, $D_2$, and $D_3$ be the digraphs given in Figure~\ref{fig:cycle}.
Clearly, $D_1$, $D_2$, and $D_3$ are orientations of $K_{1, 1, 3}$, $K_{1,1,1,2}$, and $K_{1,1,1,1,1}$, respectively.
In addition, $\NNN(D_i) \cong C_5$ for each $i=1$, $2$, and $3$.
Thus $(C_5,3)$, $(C_5,4)$, and $(C_5, 5)$ are niche-realizable.
Now let $D_4$ be a digraph with the vertex set $V(D_4)=\{v_0, v_1, v_2, v_3, v_4, v_5\}$ and the arc set \[
A(D_4)=\{(v_{i-2}, v_i), (v_{i-1}, v_i), (v_i, v_{i+1}), (v_i, v_{i+2}) \mid i \in \{0,1,2,3,4,5\}\}
\]
where all the subscripts are reduced to modulo $6$ (see Figure~\ref{fig:cycle} for an illustration).
Since each vertex $v_i$ takes $v_{i+1}$ and $v_{i+2}$ as its out-neighbors and $v_{i-1}$ and $v_{i-2}$ as its in-neighbors, $D_4$ is an orientation of $K_{2,2,2}$ with partite sets $\{v_0, v_3\}$, $\{v_1, v_4\}$, and $\{v_2, v_5\}$.
Furthermore, it is easy to see that $\NNN(D_4) \cong C_6$.
Hence the ``if'' part is true.

Suppose that $(C_n, k)$ is niche-realizable.
By Theorem~\ref{thm:diameter}, $n \le 6$.
Thus we need to show that $(n, k) \notin \{(3,3), (4, 3), (4, 4), (6, 4), (6, 5), (6, 6)\}$.
By Lemma~\ref{lem:K3}, $(n, k) \neq (3,3)$.
In addition, by lemma~\ref{prop:triangle-free-character},
$(n, k) \notin \{(6,4), (6,5), (6,6)\}$.

Suppose that $(n, k) \in \{(4,3), (4,4)\}$.
Then there is a $k$-partite tournament $D_5$ such that $\NNN(D_5) \cong C_4$ and so $\NNN(D_5)$ is triangle-free. Therefore
\begin{equation}\label{eqn:cycle degree}
d_{D_5}^+(x) \le 2 \quad \text{and} \quad d_{D_5}^-(x) \le 2
\end{equation}
for all $x \in V(D_5)$.
Let $X_1$, $\ldots$, $X_k$ be the partite sets of $D_5$.
We take $x_i \in X_i$ for each $i=1$, $2$, and $3$.
Let $x_4$ be the vertex of $D_5$ that does not belong to $\{x_1, x_2, x_3\}$.
Suppose that the subdigraph of $D_5$ induced by $\{x_1, x_2, x_3\}$ is a directed cycle.
Then, by Lemma~\ref{lem:inverse}, \eqref{eqn:cycle degree}, and the symmetry of the directed cycle, we may assume that
\[
A(D_5) \subset \{(x_1, x_2), (x_2, x_3), (x_3, x_1), (x_1, x_4), (x_2, x_4), (x_4, x_3)\}.
\]
Then, by Lemma~\ref{lem:subgraph}, $\NNN(D_5)$ is a subgraph of $P_4$ and we reach a contradiction.
Thus the subdigraph of $D_5$ induced by $\{x_1, x_2, x_3\}$ is not a directed cycle.
Then, without loss of generality, we may assume that $(x_1, x_2), (x_1, x_3), (x_2, x_3) \in A(D_5)$.
By \eqref{eqn:cycle degree}, $(x_1, x_4) \notin A(D_5)$ and $(x_4, x_3) \notin A(D_5)$.
Thus $A(D_5) \subset \{(x_1, x_2), (x_1, x_3), (x_2, x_3), (x_4, x_1), (x_3, x_4), (x_2, x_4)\}$ or $A(D_5) \subset \{(x_1, x_2), (x_1, x_3), (x_2, x_3), (x_4, x_1), (x_3, x_4), (x_4, x_2)\}$.
In both cases, $\NNN(D_5)$ is a subgraph of $P_4$ by Lemma~\ref{lem:subgraph} and we reach a contradiction.
Thus $(n, k) \notin \{(4,3), (4,4)\}$.
This completes the proof.
\end{proof}

\begin{figure}
\begin{center}
\quad\quad
\begin{tikzpicture}[x=1.0cm, y=1.0cm]

   \vertex (x1) at (0,0) [label=above:$$]{};
   \vertex (x2) at (0,1) [label=above:$$]{};
   \vertex (x3) at (0,2) [label=above:$$]{};
   \vertex (x4) at (1,0.5) [label=above:$$]{};
   \vertex (x5) at (1,1.5) [label=above:$$]{};
   \path
   (x1) edge [->,thick] (x4)
   (x1) edge [->,thick] (x5)
   (x4) edge [->,thick] (x2)
   (x5) edge [->,thick] (x2)
   (x3) edge [->,thick] (x4)
   (x5) edge [->,thick] (x3)
   (x4) edge [->,thick] (x5)	
	;
\draw (0.5, -0.5) node{$D_1$};
\end{tikzpicture}
\qquad  \qquad \qquad  \qquad
\begin{tikzpicture}[x=1.0cm, y=1.0cm]

  \vertex (x1) at (0,0) [label=above:$$]{};
  \vertex (x2) at (0,1) [label=above:$$]{};
  \vertex (x3) at (0,2) [label=above:$$]{};
  \vertex (x4) at (1,0.5) [label=above:$$]{};
  \vertex (x5) at (1,1.5) [label=above:$$]{};
   \path
   (x1) edge [-,bend left=50,thick] (x3)
   (x2) edge [-,thick] (x3)
   (x2) edge [-,thick] (x5)
   (x4) edge [-,thick] (x5)
   (x1) edge [-,thick] (x4)

	;
\draw (0.5, -0.5) node{$\NNN(D_1)$};
\end{tikzpicture}

\vskip1cm
\quad
\begin{tikzpicture}[x=1.0cm, y=1.0cm]

   \vertex (x1) at (0,0) [label=above:$$]{};
   \vertex (x2) at (0,1) [label=above:$$]{};
   \vertex (x3) at (0,2) [label=above:$$]{};
   \vertex (x4) at (1,0.5) [label=above:$$]{};
   \vertex (x5) at (1,1.5) [label=above:$$]{};
   \path
   (x4) edge [->,thick] (x1)
   (x1) edge [->,thick] (x2)
   (x5) edge [->,thick] (x1)
   (x1) edge [->,bend left=50,thick] (x3)

   (x2) edge [->,thick] (x4)
   (x2) edge [->,thick] (x5)
   (x3) edge [->,thick] (x2)

   (x3) edge [->,thick] (x5)
   (x4) edge [->,thick] (x3)
	;
\draw (0.5, -0.5) node{$D_2$};
\end{tikzpicture}
\qquad  \qquad \qquad  \qquad
\begin{tikzpicture}[x=1.0cm, y=1.0cm]

  \vertex (x1) at (0,0) [label=above:$$]{};
  \vertex (x2) at (0,1) [label=above:$$]{};
  \vertex (x3) at (0,2) [label=above:$$]{};
  \vertex (x4) at (1,0.5) [label=above:$$]{};
  \vertex (x5) at (1,1.5) [label=above:$$]{};
   \path
   (x1) edge [-,bend left=50,thick] (x3)
   (x2) edge [-,thick] (x3)
   (x2) edge [-,thick] (x5)
   (x4) edge [-,thick] (x5)
   (x1) edge [-,thick] (x4)

	;
\draw (0.5, -0.5) node{$\NNN(D_2)$};
\end{tikzpicture}
\vskip1cm
\quad
\begin{tikzpicture}[x=1.0cm, y=1.0cm]

  \vertex (x1) at (0,0) [label=above:$$]{};
  \vertex (x2) at (0,1) [label=above:$$]{};
  \vertex (x3) at (0,2) [label=above:$$]{};
  \vertex (x4) at (1,0.5) [label=above:$$]{};
  \vertex (x5) at (1,1.5) [label=above:$$]{};
   \path
   (x2) edge [->,thick] (x1)
   (x3) edge [->,bend right=50,thick] (x1)
   (x1) edge [->,thick] (x4)
   (x1) edge [->,thick] (x5)

   (x2) edge [->,thick] (x3)
   (x4) edge [->,thick] (x2)
   (x5) edge [->,thick] (x2)

   (x4) edge [->,thick] (x3)
   (x3) edge [->,thick] (x5)

   (x5) edge [->,thick] (x4)
	
	;
\draw (0.5, -0.5) node{$D_3$};
\end{tikzpicture}
\qquad  \qquad \qquad  \qquad
\begin{tikzpicture}[x=1.0cm, y=1.0cm]

  \vertex (x1) at (0,0) [label=above:$$]{};
  \vertex (x2) at (0,1) [label=above:$$]{};
  \vertex (x3) at (0,2) [label=above:$$]{};
  \vertex (x4) at (1,0.5) [label=above:$$]{};
  \vertex (x5) at (1,1.5) [label=above:$$]{};
   \path
   (x2) edge [-,thick] (x3)
   (x4) edge [-,thick] (x5)
   (x1) edge [-,bend left=50,thick] (x3)
   (x2) edge [-,thick] (x4)
   (x1) edge [-,thick] (x5)

	;
\draw (0.5, -0.5) node{$\NNN(D_3)$};
\end{tikzpicture}
\vskip1cm
\quad
\begin{tikzpicture}[x=1.0cm, y=1.0cm]

  \vertex (x1) at (0,0) [label=above:$$]{};
  \vertex (x2) at (0,1) [label=above:$$]{};
  \vertex (x3) at (0,2) [label=above:$$]{};
  \vertex (x4) at (1,0) [label=above:$$]{};
  \vertex (x5) at (1,1) [label=above:$$]{};
  \vertex (x6) at (1,2) [label=above:$$]{};
   \path
    (x3) edge [->,thick] (x2)
    (x3) edge [->,bend right=50,thick] (x1)
    (x6) edge [->,thick] (x5)
    (x6) edge [->,bend left=50,thick] (x4)
    (x2) edge [->,thick] (x1)
    (x2) edge [->,thick] (x6)
    (x5) edge [->,thick] (x3)
    (x5) edge [->,thick] (x4)
    (x4) edge [->,thick] (x2)
    (x4) edge [->,thick] (x3)
    (x1) edge [->,thick] (x6)
    (x1) edge [->,thick] (x5)
%
%
%
	;
\draw (0.5, -0.5) node{$D_4$};
\end{tikzpicture}
\qquad  \qquad \qquad  \qquad
\begin{tikzpicture}[x=1.0cm, y=1.0cm]

  \vertex (x1) at (0,0) [label=above:$$]{};
  \vertex (x2) at (0,1) [label=above:$$]{};
  \vertex (x3) at (0,2) [label=above:$$]{};
  \vertex (x4) at (1,0) [label=above:$$]{};
  \vertex (x5) at (1,1) [label=above:$$]{};
  \vertex (x6) at (1,2) [label=above:$$]{};
   \path
   (x3) edge [-,thick] (x2)
   (x2) edge [-,thick] (x1)
   (x1) edge [-,thick] (x6)
   (x6) edge [-,thick] (x5)
   (x5) edge [-,thick] (x4)
   (x4) edge [-,thick] (x3)
	;
\draw (0.5, -0.5) node{$\NNN(D_4)$};
\end{tikzpicture}

\end{center}
\caption{The digraphs $D_1$, $D_2$, $D_3$, and $D_4$ which are isomorphic to some orientations of $K_{1,1,3}$, $K_{1,1,1,2}$, $K_{1,1,1,1,1}$, and $K_{2,2,2}$, respectively, and their niche graphs}
\label{fig:cycle}
\end{figure}
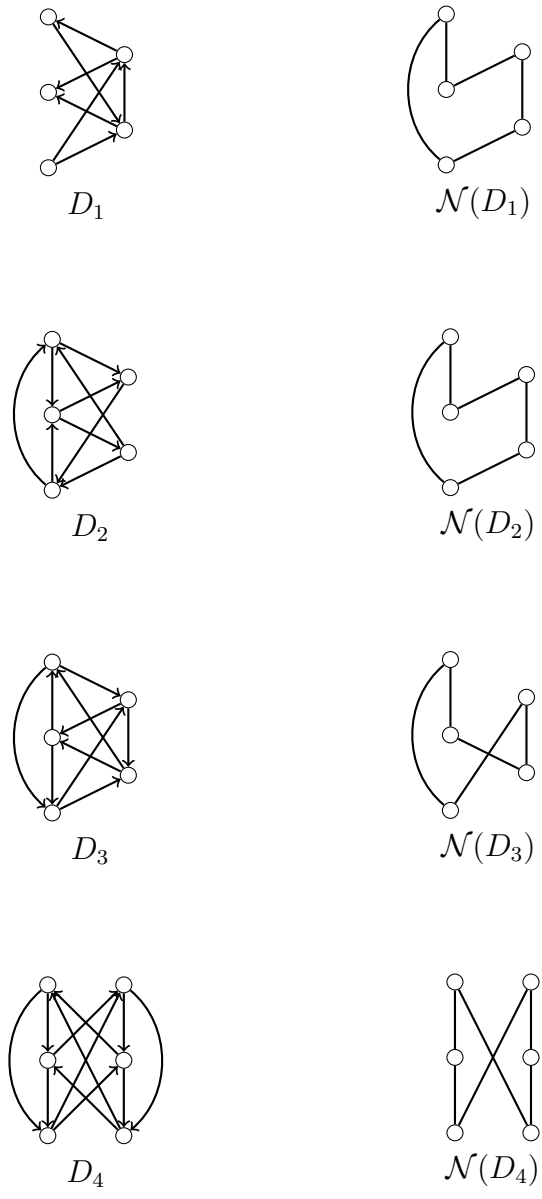

\begin{Lem} \label{thm:tree-character-graph}
Let $G$ be a connected triangle-free graph with $3 \leq |V(G)| \leq 5$, stability number at most $3$, and diameter at most $4$.
Then the following are true:
\begin{enumerate}
\item[(1)] each vertex in $G$ has degree at most $3$;
\item[(2)] $G$ is isomorphic to a path $P_i$ for some $i\in \{3,4,5\}$ or cycle $C_j$ for some $j\in \{4,5\}$ or the graph $G_k$ for some $k \in \{1,2,3,4\}$ given in Figure~\ref{fig:graphs}.
\end{enumerate}
\end{Lem}

\begin{proof}
To show the statement (1) by contradiction,
suppose that there exists a vertex $x$ in $G$ of degree at least $4$.
Then there exist four distinct vertices $x_1$, $x_2$, $x_3$, and $x_4$ which are adjacent to $x$ in $G$.
Since $G$ is triangle-free,
$x_i$ and $x_j$ are not adjacent if $i \neq j$.
Therefore $\{x_1,x_2,x_3,x_4\}$ is a stable set, which contradicts the hypothesis that $G$ has stability number at most $3$.
Thus the statement (1) is true.

To show the statement (2),
we first consider the case where $G$ is a tree.
If $G$ is isomorphic to a path, then $G \cong P_i$ for some $i\in \{3,4,5\}$ by the hypothesis.
Suppose that $G$ is not a path graph.
Let $t$ be a diameter of $G$.
Then $t \leq 4$ by the hypothesis and there exists an induced path $P:=x_1\ldots x_{t+1}$ of length $t$ in $G$.
Since $G$ is not a path graph, there exist a vertex of degree at least $3$ on $P$.
Let $x_i$ be a vertex of degree at least $3$.
Then $x_i$ has degree $3$ by the statement (1).
By the choice of $P$, $i \neq 1 $ and $i \neq t+1$.
If $t=1$, then $G$ is a complete, which is contradiction.
Therefore $t \geq 2$.
If $t=2$,
then $i=2$ and so $G$ is isomorphic to $G_1$ given in Figure~\ref{fig:graphs}.
Suppose $t=3$. Then $i=2 $ or $3$.
By symmetry, we may assume $i=2$.
Then there exists a vertex $x_5$ not on $P$ which is adjacent to $x_2$.
Since $|V(G)| \leq 5$ by the hypothesis,
$V(G)=\{x_1,x_2,x_3,x_4,x_5\}$.
Then, since $G$ is a tree, $x_2$ is the only vertex adjacent to $x_5$ in $G$.
Thus $G$ isomorphic to $G_2$ given in Figure~\ref{fig:graphs}.
If $t=4$, then $G=P$, which is a contradiction.

Now we consider the case where $G$ is not a tree.
Then $G$ has a cycle $C$ of length at least $4$ since $G$ is triangle-free and connected.
Then $4 \leq |V(G)|$.
If $|V(G)|=4$, then $G=C$, so $G$ is isomorphic to a cycle $C_4$ by the hypothesis that $G$ is triangle-free.
Suppose that $|V(G)|=5$.
If $G$ is a cycle, then $G$ is isomorphic to a cycle $C_5$ by the hypothesis.
Now we suppose that $G$ is not a cycle.
If $|V(C)|=5$, then $C$ is a spanning subgraph of $G$ and so $C$ has a chord, which contradicts the hypothesis that $G$ is triangle-free.
Therefore $|V(C)|=4$.
Let $y$ be the vertex in $V(G)\setminus V(C)$.
Then there exists a vertex $y'$ on $C$ which is adjacent to $y$ by the hypothesis that $G$ is connected.
Therefore $y'$ has degree $3$ by the statement (1).
If $y$ has degree $3$, then it is easy to check that $G$ contains a triangle, which is a contradiction.
Therefore $y$ has degree $1$ or $2$.
If $y$ has degree $1$, then $G$ is isomorphic to a graph $G_3$ given in Figure~\ref{fig:graphs}.
If $y$ has degree $2$, then $G$ is isomorphic to a graph $G_4$ given in Figure~\ref{fig:graphs}.
Therefore we have shown that the statement (2) is true.
\end{proof}
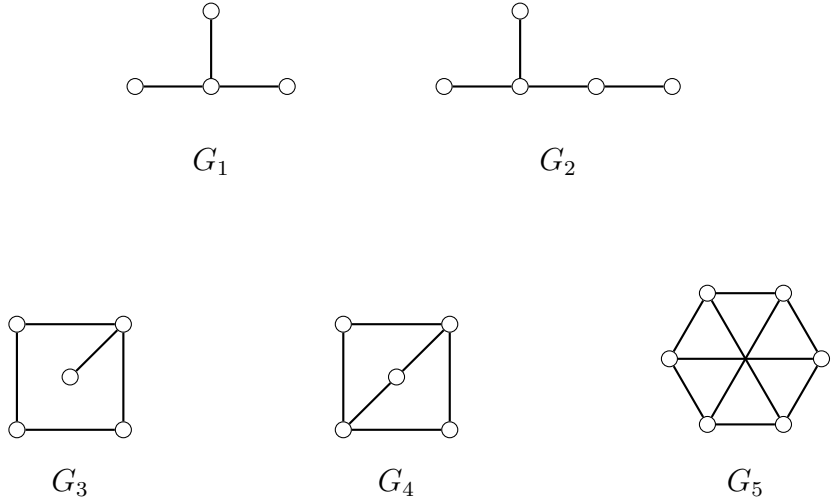
\begin{figure}
\begin{center}
    \begin{tikzpicture}[x=1.0cm, y=1.0cm]

  \vertex (x1) at (0,0) [label=above:$$]{};
  \vertex (x2) at (1,0) [label=above:$$]{};
  \vertex (x3) at (-1,0) [label=above:$$]{};
  \vertex (x4) at (0,1) [label=above:$$]{};
   \path
   (x1) edge [-,thick] (x2)
   (x1) edge [-,thick] (x3)
   (x1) edge [-,thick] (x4)
;
\draw (0, -1) node{$G_1$};
\end{tikzpicture}
\qquad \qquad
   \begin{tikzpicture}[x=1.0cm, y=1.0cm]
  \vertex (x1) at (0,0) [label=above:$$]{};
  \vertex (x2) at (1,0) [label=above:$$]{};
  \vertex (x3) at (-1,0) [label=above:$$]{};
  \vertex (x4) at (0,1) [label=above:$$]{};
  \vertex (x5) at (2,0) [label=above:$$]{};
   \path
   (x1) edge [-,thick] (x2)
   (x1) edge [-,thick] (x3)
   (x1) edge [-,thick] (x4)
(x2) edge [-,thick] (x5)
;
\draw (0.5, -1) node{$G_2$};
\end{tikzpicture}
\vskip1cm
\quad
\begin{tikzpicture}[x=1.0cm, y=1.0cm]
  \vertex (x1) at (0,0) [label=above:$$]{};
  \vertex (x2) at (0.7,0.7) [label=above:$$]{};
  \vertex (x3) at (-0.7,0.7) [label=above:$$]{};
  \vertex (x4) at (-0.7,-0.7) [label=above:$$]{};
  \vertex (x5) at (0.7,-0.7) [label=above:$$]{};
   \path
   (x1) edge [-,thick] (x2)
   (x2) edge [-,thick] (x3)
   (x3) edge [-,thick] (x4)
(x4) edge [-,thick] (x5)
(x5) edge [-,thick] (x2);
\draw (0, -1.4) node{$G_3$};
\end{tikzpicture}
\qquad \qquad \qquad
\begin{tikzpicture}[x=1.0cm, y=1.0cm]
  \vertex (x1) at (0,0) [label=above:$$]{};
  \vertex (x2) at (0.7,0.7) [label=above:$$]{};
  \vertex (x3) at (-0.7,0.7) [label=above:$$]{};
  \vertex (x4) at (-0.7,-0.7) [label=above:$$]{};
  \vertex (x5) at (0.7,-0.7) [label=above:$$]{};
   \path
   (x1) edge [-,thick] (x2)
   (x1) edge [-,thick] (x4)
   (x2) edge [-,thick] (x3)
   (x3) edge [-,thick] (x4)
(x4) edge [-,thick] (x5)
(x5) edge [-,thick] (x2);
\draw (0, -1.4) node{$G_4$};
\end{tikzpicture}
\qquad \qquad \qquad
\begin{tikzpicture}[x=1.0cm, y=1.0cm]
  \vertex (x1) at (1,0) [label=above:$$]{};
  \vertex (x2) at (0.5,0.87) [label=above:$$]{};
  \vertex (x3) at (-0.5,0.87) [label=above:$$]{};
  \vertex (x4) at (-1,0) [label=above:$$]{};
  \vertex (x5) at (-0.5,-0.87) [label=above:$$]{};
   \vertex (x6) at (0.5,-0.87) [label=above:$$]{};
   \path
   (x1) edge [-,thick] (x2)
   (x1) edge [-,thick] (x4)
   (x2) edge [-,thick] (x3)
   (x2) edge [-,thick] (x5)
   (x3) edge [-,thick] (x4)
    (x3) edge [-,thick] (x6)
(x4) edge [-,thick] (x5)
(x5) edge [-,thick] (x6)
(x6) edge [-,thick] (x1);
\draw (0, -1.64) node{$G_5$};
\end{tikzpicture}
    \end{center}
\caption{Connected triangle-free graphs mentioned in Lemmas~\ref{thm:tree-character-graph} and~\ref{prop:six-vertices-char}}
 \label{fig:graphs}
 \end{figure}

\begin{Lem} \label{prop:six-vertices-char}
Let $G$ be a connected triangle-free graph with six vertices. Then $(G,k)$ is niche-realizable for some integer $k \geq 3$ if and only if $k=3$ and $G$ is isomorphic to the cycle $C_6$ or the graph $G_5$ given in Figure~\ref{fig:graphs}.
\end{Lem}

\begin{proof}
Suppose that $(G,k)$ is niche-realizable for some integer $k \geq 3$.
Then there exists a $k$-partite tournament $D$ such that ${\mathcal N}(D) \cong G$.
Since $|V(G)|=6$, $k=3$ by Lemma~\ref{prop:triangle-free-character}.
We denote the partite sets of $D$ by $(X_1,X_2,X_3)$.
If $|X_l|=1$ for some $l \in \{1,2,3\}$, then $d^+_D(x)+d^-_D(x) = 5$ for the vertex $x$ in $X_l$ by Lemma~\ref{lem:in and out},
which contradicts Lemma~\ref{lem:in-out-degree-condition}.
 Therefore each partite set in $D$ has at least size $2$.
Since $|V(G)|=6$ and $k=3$, each partite set in $D$ has size $2$. Therefore $d^+_D(v)+d^-_D(v)=4$ by Lemma~\ref{lem:in and out} and so, by Lemma~\ref{lem:in-out-degree-condition},
 \begin{equation}
\label{eq:prop:six-vertices-char}
d^+_D(v)=d^-_D(v)=2
\end{equation}
 for all $v \in V(D)$.
 Now let $X_1=\{v_1,v_2\}$, $X_2=\{v_3,v_4\}$, and $X_3=\{v_5,v_6\}$.

 {\it Case 1}. The two vertices in $X_i$ are not adjacent in $G$ for each $i=1$, $2$, and $3$.
 Then the out-neighbors (resp.\ in-neighbors) of each vertex belong to distinct partite sets.
Now, without loss of generality,
we may assume $N^+_D(v_1)=\{v_3,v_5\}$ and $N^-_D(v_1)=\{v_4,v_6\}$.
By symmetry, we may assume that $(v_3,v_5)\in A(D)$.
Then $
N^-_D(v_5)=\{v_1,v_3\}$, so $N^+_D(v_5)=\{v_2,v_4\}$.
By the case assumption, $(v_3,v_6)\notin A(D)$, so $(v_6,v_3)\in A(D)$.
Then $N^-_D(v_3)=\{v_1,v_6\}$, so $N^+_D(v_3)=\{v_2,v_5\}$.
Therefore $
N^-_D(v_2)=\{v_3,v_5\}$ and $ N^+_D(v_2)=\{v_4,v_6\}$.
Thus $N^-_D(v_4)=\{v_2,v_5\}$ and $N^+_D(v_4)=\{v_1,v_6\}$.
Hence
$N^-_D(v_6)=\{v_2,v_4\}$ and $N^+_D(v_6)=\{v_1,v_3\}$.
Now $D$ is uniquely determined and isomorphic to $D_4$ given in Figure~\ref{fig:cycle} whose niche graph is a cycle of length $6$.
{\it Case 2}. The two vertices in $X_j$ are adjacent in $G$ for some $j \in \{1,2,3\}$.
 Without loss of generality, we may assume that $j=2$. By symmetry and Lemma~\ref{lem:inverse}, we may assume $\{v_3,v_4\} \subset N^+_D(v_1)$. Then \begin{equation}
\label{eq:prop:six-vertices-char-1}
N^+_D(v_1)=\{v_3,v_4\}
\end{equation}
 and $N^-_D(v_1)=\{v_5,v_6\}$ by \eqref{eq:prop:six-vertices-char}.
 If $N^+_D(v_2)=\{v_3,v_4\}$, then $v_1$ and $v_2$ are true twins and so, by Lemma~\ref{lem:equivalence contains a triangle}, $G$ contains a triangle, which contradicts the hypothesis that $G$ is triangle-free.
 Therefore $N^+_D(v_2) \neq \{v_3,v_4\}$ and so
 $N^-_D(v_2) \cap \{v_3,v_4\} \neq \emptyset$.
 Then, there are two subcases to consider: $N^-_D(v_2)\cap \{v_3,v_4\}= \{v_3,v_4\}$; $ |N^-_D(v_2)\cap \{v_3,v_4\}| =1$.

 {\it Subcase 1.} $N^-_D(v_2)\cap \{v_3,v_4\}= \{v_3,v_4\}$. Then $N^-_D(v_2)= \{v_3,v_4\}$ and  $N^+_D(v_2)= \{v_5,v_6\}$ by~\eqref{eq:prop:six-vertices-char}, so \begin{equation}
\label{eq:prop:six-vertices-char-2} v_5v_6 \in E(G).
\end{equation}
   Moreover, $|N^-_D(v_3) \cap \{v_5,v_6\}|=|N^-_D(v_4) \cap \{v_5,v_6\}|=1$ by~\eqref{eq:prop:six-vertices-char}.
    If $N^-_D(v_3) \cap N^-_D(v_4) \cap \{v_5,v_6\} \neq \emptyset$, then $v_3$ and $v_4$ are true twins, which is a contradiction.
  Therefore $N^-_D(v_3) \cap N^-_D(v_4) \cap \{v_5,v_6\} =\emptyset$.
  By symmetry, we may assume
that $N^-_D(v_3) \cap \{v_5,v_6\}=\{v_5\}$.
Then $N^-_D(v_4) \cap \{v_5,v_6\}=\{v_6\}$.
Therefore $\{v_1v_5, v_1v_6\} \subset E(G)$ and so, by~\eqref{eq:prop:six-vertices-char-2}, $v_1v_5v_6v_1$ is a triangle in $G$, which contradicts the hypothesis.

 {\it Subcase 2.}  $ |N^-_D(v_2)\cap \{v_3,v_4\}| =1$. By symmetry, we may assume $N^-_D(v_2)\cap \{v_3,v_4\}= \{v_4\}$.
 Then $(v_2,v_3)\in A(D)$.
 Therefore $N^-_D(v_3)=\{v_1,v_2\}$ by~\eqref{eq:prop:six-vertices-char-1} and so, by~\eqref{eq:prop:six-vertices-char}, $N^+_D(v_3)=\{v_5,v_6\}$.
Moreover, $|N^+_D(v_2) \cap \{v_5,v_6\}|=1$.
By symmetry, we may assume $(v_2,v_5) \in A(D)$.
Then $N^+_D(v_2)=\{v_3,v_5\}$. Therefore $N^-_D(v_2)=\{v_4,v_6\}$ and $N^-_D(v_5)=\{v_2,v_3\}$.
Thus $N^+_D(v_5)=\{v_1,v_4\}$.
Hence $N^-_D(v_4)=\{v_1,v_5\}$ and $N^+_D(v_4)=\{v_2,v_6\}$.
Then $N^-_D(v_6)=\{v_3,v_4\}$ and $N^+_D(v_6)=\{v_1,v_2\}$.
Now $D$ is uniquely determined.
It is easy to check that ${\mathcal N}(D)$ is isomorphic to the graph $G_5$ given in Figure~\ref{fig:graphs}.
Therefore the ``only if" part is true.

By the way, $3$-partite tournaments whose niche graphs are isomorphic to the cycle $C_6$ and the graph $G_5$ were constructed in Cases $1$ and $2$, respectively.
Thus  the ``if" part is true and this completes the proof.
\end{proof}

Now we are ready to characterize connected triangle-free niche-realizable graphs.

\begin{Thm} \label{thm:triangle-free}
Let $G$ be a connected triangle-free graph with at least three vertices.
Then $(G,k)$ is niche-realizable for some integer $k \geq 3$ if and only if $k \in \{3,4,5\}$ and $G$ is isomorphic to a graph belonging to the following set:
\[
 \begin{cases}\{ P_3, P_4, P_5, C_5, C_6, G_4, G_5 \} & \mbox{if $k=3$;} \\
 \{ P_4, C_5 \}  & \mbox{if $k=4$;} \\
 \{ C_5 \}& \mbox{if $k=5$} \end{cases}
 \]
 where $G_4$ and $G_5$ are the graphs given in Figure~\ref{fig:graphs}.
 \end{Thm}

 \begin{proof}
 Let $n$ denote the number of vertices in $G$.
 To show the ``only if" part, suppose that $(G,k)$ is niche-realizable for some integer $k \geq 3$.
Then there exists a $k$-partite tournament $D$ such that ${\mathcal N}(D) \cong G$.
 By Lemma~\ref{prop:triangle-free-character},
 $k \leq 5$ and $n \leq 6$.
 If $n = 6$, then $k=3$ and $G$ is isomorphic to a cycle $C_6$ or the graph $G_5$ given in Figure~\ref{fig:graphs} by Lemma~\ref{prop:six-vertices-char}.
 Now we suppose that $n\leq 5$.
 If $G$ is a path or a cycle, then, by Lemmas~\ref{lem:path} and \ref{lem:cycle},
  $G$ is isomorphic to $P_3$, $P_4$, $P_5$, or $C_5$ when $k=3$; $G$ is isomorphic to $P_4$ or $C_5$ when $k=4$;  $G$ is isomorphic to $C_5$ when $k=5$.

  Now we suppose that $G$ is neither a path nor a cycle.
 By Theorem~\ref{thm:stability number} and Corollary~\ref{cor:diameter}, $G$ has stability number at most $3$ and diameter at most $4$.
 Therefore, by Lemma~\ref{thm:tree-character-graph}, $G$ is isomorphic to the graph $G_j$ given in Figure~\ref{fig:graphs} for some $j \in \{1,2,3,4\}$.
 Thus it remains to show that $k=3$ and $G \cong G_4$.
Since $G$ is neither a path nor a cycle, there exists a vertex $v_1$ of degree at least $3$ in $G$.
If $v_1$ has degree at least $4$, then $G \not \cong G_i$ for each $1\leq i \leq 4$.
Therefore $v_1$ has degree $3$.
 Since each of $v_1$ and its neighbors has indegree at most $2$ and outdegree at most $2$ by Lemma~\ref
{lem:in-out-degree-condition},
$v_1$ is adjacent to at most two vertices if $d^+_D(v_1) = 0$ or $d^-_D(v_1)= 0$, which is a contradiction.
Therefore $d^+_D(v_1)\geq 1$ and $d^-_D(v_1) \geq 1$.
If $d^+_D(v_1) =1$ and $d^-_D(v_1) = 1$, then $v_1$ has degree at most $2$ for the same reason as the previous one, which is a contradiction.
Therefore $d^+_D(v_1)  \geq 2 $ or $d^-_D(v_1)  \geq 2 $ and so $3 \leq d^+_D(v_1)+d^-_D(v_1)$.
By Lemma~\ref{lem:inverse}, we may assume $d^+_D(v_1) \geq 2$ and then, by Lemma~\ref{lem:in-out-degree-condition},  $d^+_D(v_1) =2$.
Now we let \begin{equation}
\label{eq:thm:triangle-free1}
N^+_D(v_1)=\{v_3,v_4\} \end{equation}
and $v_5$ be an in-neighbor of $v_1$ in $D$.
 Suppose $n \leq 4$.
 Then $n=4$ since degree of $v_1$ is $3$.
 Therefore $G$ is isomorphic to the graph $G_1$. However, two neighbors $v_3$ and $v_4$ of $v_1$ are adjacent in $G$ by~\eqref{eq:thm:triangle-free1}, which is a contradiction.
Thus $n=5$ and so
\[G \cong \text{$G_2$, $G_3$, or $G_4$.}\]
Let $v_2$ to be a vertex of $G$ other than $v_1$, $v_3$, $v_4$, and $v_5$.
Let $X_i$ be the partite sets of $D$ for each $ 1 \leq i \leq k$.
We may assume that $v_1 \in X_1$.
Since $k \geq 3$ and $n=5$, $|X_1|=1$, $2$, or $3$.
Since $d^-_D(v_1) \geq 1 $ and $d^+_D(v_1) \geq 2$, $|X_1|=1$ or $2$.
Suppose, to the contrary, that $|X_1|=1$.
Then $X_1=\{v_1\}$, so $N^-_D(v_1)=\{v_2,v_5\}$ and then $v_2v_5 \in E(G)$.
By~\eqref{eq:thm:triangle-free1}, $v_3v_4 \in E(G)$, so $G-v_1$ has at least two edges $v_3v_4$ and $v_2v_5$ not sharing end points in $G$, which cannot happen in any of $G_2$, $G_3$ and $G_4$.
Thus $|X_1|=2$ and \[X_1=\{v_1,v_2\}.\]
Then, since $v_1$ has three neighbors which form a stable set, each of $v_3$, $v_4$, and $v_5$ should be a common out-neighbor or in-neighbor of $v_1$ and a vertex adjacent to $v_1$.
By the way, $v_3$ and $v_4$ are common out-neighbors and $v_5$ is a common in-neighbor by~\eqref{eq:thm:triangle-free1}.
Therefore $N^-_D(v_3)$, $N^-_D(v_4)$, and $N^+_D(v_5)$ are $2$-element sets which differ from each other.
Furthermore, since $v_3v_4 \in E(G)$,
one of $v_3$ and $v_4$ is not adjacent to $v_1$.
Without loss of generality, we may assume that $v_3$ is the vertex not adjacent to $v_1$.

Suppose, to the contrary, that $v_3$ and $v_4$ are in different partite sets.
Since $v_1$ and $v_3$ are not adjacent in $G$, $(v_4,v_3) \in A(D)$ by \eqref{eq:thm:triangle-free1}. Then $N^-_D(v_3)=\{v_1,v_4\}$ by Lemma~\ref{lem:in-out-degree-condition}.
Since $v_5$ is adjacent to $v_1$, $v_1$ and $v_5$ have common in-neighbor or out-neighbor.
Since $N^-_D(v_1)=\{v_5\}$, $v_1$ and $v_5$ cannot have any common in-neighbor and so have a common out-neighbor.
Since $N^+_D(v_1)=\{v_3,v_4\}$, $v_3$ and $v_4$ are possible common out-neighbors of $v_1$ and $v_5$.
However, $v_3$ already has two in-neighbors distinct from $v_5$. Therefore $v_4$ must be a common out-neighbor of $v_1$ and $v_5$.
Thus
 $(v_5,v_4) \in A(D)$ and so, by Lemma~\ref{lem:in-out-degree-condition}, $N^+_D(v_5)=\{v_1,v_4\}$.
Therefore $N^-_D(v_3)=N^+_D(v_5)$, which is a contradiction.
Thus $v_3$ and $v_4$ belong to the same partite set and $k=3$.
Let $X_2=\{v_3,v_4\}$ and $X_3=\{v_5\}$.
Then, since $v_1$ and $v_3$ are not adjacent in $G$, $(v_3,v_5) \in A(D)$.
Since $d^-_D(v_3)=2$, $(v_2,v_3)\in A(D)$ and so $N^-_D(v_3)=\{v_1,v_2\}$.
Since $N^-_D(v_3)\neq N^-_D(v_4)$ and $d^-_D(v_4)=2$,
$(v_5,v_4) \in A(D)$.
Therefore $N^-_D(v_4)=\{v_1,v_5\}$ and so $N^+_D(v_4)=\{v_2\}$.
Moreover, since $N^+_D(v_5)=\{v_1,v_4\}$, $(v_2,v_5) \in A(D)$.
Now $D$ is uniquely determined. Then, it is easy to check that ${\mathcal N}(D) \cong G_4$. Therefore the ``only if" part is true.

The pairs $(P_3,3)$, $(P_4,3)$, $(P_5,3)$ and $(P_4,4)$ are niche-realizable by Lemma~\ref{lem:path}.
The pairs $(C_5,3)$, $(C_5,4)$, $(C_5,5)$, and $(C_6,3)$ are
niche-realizable by Lemma~\ref{lem:cycle}.
The pair $(G_5,3)$ is niche-realizable by Lemma~\ref{prop:six-vertices-char}.
The pair $(G_4,3)$ is niche-realizable as we have constructed a $3$-partite tournament $D$ whose niche graph is isomorphic to $G_4$ while showing the ``only if" part of the statement.
Hence the ``if" part is true.
 \end{proof}


\end{document}